\documentclass[11pt,a4paper,dvipsnames]{article}

\usepackage{url,amsmath,amssymb,latexsym,pstricks,mathrsfs,comment,amsthm,graphicx,tikz,tikz-cd,enumerate,accents,pgffor,cite,wrapfig,multicol,float,cases,calc,geometry}
\usepackage[colorlinks]{hyperref}

\usepackage{color}

\usepackage{enumitem}

\usepackage[T1]{fontenc}
\usepackage{wasysym}
\usepackage{stmaryrd}

\usetikzlibrary{calc}

\newcommand\nc\newcommand
\nc\rnc\renewcommand

\newcommand{\JE}[1]{{\color{blue}\sf [JE: #1]}}
\nc\JEnew[1]{\textcolor{magenta}{#1}}

\allowdisplaybreaks

\rnc\emptyset\varnothing

\renewcommand{\S}{\mathcal{S}}

\nc\T{\mathcal T}
\nc\Reg{\operatorname{Reg}}
\nc\ol\overline
\nc\ul\underline
\nc\Om\Omega
\nc\Lam\Lambda
\nc\A{\mathcal A}
\nc\E{\mathcal E}
\nc\EE{\boldsymbol{\E}}
\nc\F{\mathcal F}
\nc\FF{\boldsymbol{\F}}
\nc\br{{\bf r}}
\nc\bk{{\bf k}}
\nc\bl{{\bf l}}
\nc\bL{{\bf L}}
\nc\bR{{\bf R}}
\nc\bB{{\bf B}}
\nc\bbL{\mathbb L}
\nc\bbR{\mathbb R}
\nc\Eq{\mathfrak{Eq}}
\nc\da\downarrow
\nc\Ht{\operatorname{\sf Ht}}

\nc\Zh{\includegraphics[width=3.5mm]{Zh.pdf}}
\nc\Wh{\includegraphics[width=3.8mm]{Wh.pdf}}
\nc\Dh{{\lower0.35 ex\hbox{$\includegraphics[width=2.5mm]{Dh.pdf}$}}}
\nc\Jh{{\lower0.42 ex\hbox{$\includegraphics[width=2.8mm]{Jh.pdf}$}}}

\DeclareMathSymbol{\widehatsym}{\mathord}{largesymbols}{"62}
\newcommand\lowerwidehatsym{%
  \text{\smash{\raisebox{-1.3ex}{%
    $\widehatsym$}}}}
\newcommand\fixwidehat[1]{%
  \mathchoice
    {\accentset{\displaystyle\lowerwidehatsym}{#1}}
    {\accentset{\textstyle\lowerwidehatsym}{#1}}
    {\accentset{\scriptstyle\lowerwidehatsym}{#1}}
    {\accentset{\scriptscriptstyle\lowerwidehatsym}{#1}}
}
\nc\wh\fixwidehat

\nc\trans[1]{\big(\begin{smallmatrix}#1\end{smallmatrix}\big)}
\nc\Btrans[1]{\Big(\begin{smallmatrix}#1\end{smallmatrix}\Big)}

\renewcommand{\H}{\mathrel{\mathscr H}}
\renewcommand{\L}{\mathrel{\mathscr L}}
\newcommand{\R}{\mathrel{\mathscr R}}
\newcommand{\D}{\mathrel{\mathscr D}}
\newcommand{\J}{\mathrel{\mathscr J}}
\newcommand{\K}{\mathrel{\mathscr K}}
\nc\leqR{\leq_{\R}}
\nc\leqL{\leq_{\L}}
\nc\leqJ{\leq_{\J}}
\nc\leqK{\leq_{\K}}
\nc\geqR{\geq_{\R}}
\nc\geqL{\geq_{\L}}
\nc\geqJ{\geq_{\J}}
\nc\hL{\mathrel{\wh{\mathscr L}}}
\nc\hR{\mathrel{\wh{\mathscr R}}}
\nc\hH{\mathrel{\wh{\mathscr H}}}
\nc\hJ{\mathrel{\wh{\mathscr J}}}
\nc\hD{\mathrel{\wh{\mathscr D}}}
\nc\hK{\mathrel{\wh{\mathscr K}}}

\newcommand{\Cong}{\operatorname{\sf Cong}}

\newcommand{\rank}{\operatorname{rank}}
\newcommand{\rk}{\operatorname{rank}}
\newcommand{\id}{\operatorname{id}}

\nc{\ldb}{[\hspace{-0.5truemm}[}
\nc{\rdb}{]\hspace{-0.5truemm}]}

\numberwithin{equation}{section}

\newtheorem{thm}[equation]{Theorem}
\newtheorem{lemma}[equation]{Lemma}
\newtheorem{cor}[equation]{Corollary}
\newtheorem{prop}[equation]{Proposition}

\theoremstyle{definition}

\newtheorem{rem}[equation]{Remark}

\newcommand{\restr}{{\restriction}}

\nc\LV{\mathcal L_V}
\nc\bit{\begin{itemize}}
\nc\eit{\end{itemize}}
\nc\ben{\begin{enumerate}[label=\textup{(\roman*)},leftmargin=7mm]}
\nc\bena{\begin{enumerate}[label=\textup{(\alph*)},leftmargin=7mm]}
\nc\een{\end{enumerate}}
\nc\bmc{\begin{multicols}}
\nc\emc{\end{multicols}}
\nc\ka\kappa
\nc\lam\lambda
\rnc\th\theta
\nc\al\alpha
\nc\ga\gamma
\nc\set[2]{\{#1:#2\}}
\nc\bigset[2]{\big\{#1:#2\big\}}
\nc\Bigset[2]{\Big\{#1:#2\Big\}}
\nc\im{\operatorname{im}}
\nc\LSUB{\operatorname{LSUB}}
\nc\Span{\operatorname{Span}}
\nc\Aut{\operatorname{Aut}}
\nc\GL{\operatorname{GL}}
\nc\drank{\operatorname{drank}}
\nc\codim{\operatorname{codim}}
\nc\es\varnothing
\nc\nab\nabla
\nc\normal\unlhd
\nc\si\sigma
\nc\ve\varepsilon
\nc\ze\zeta
\nc\be\beta
\nc\mt\mapsto
\nc\sub\subseteq
\nc\sm\setminus
\nc\mr\mathrel
\nc\De\Delta
\nc\Si\Sigma
\nc\Th\Theta
\nc\B{\mathscr B}
\nc\cC{\mathcal{C}}
\nc\cP{\mathcal{P}}
\nc\bP{\mathbf{P}}
\nc\iR{\mathfrak{R}}
\nc\iL{\mathfrak{L}}
\nc\iK{\mathfrak{K}}
\nc\bI{\mathbf{I}}
\nc\bJ{\mathbf{J}}
\nc\de\delta
\nc\del\partial

\nc\AND{\qquad\text{and}\qquad}
\nc\WHERE{\qquad\text{where}\qquad}
\nc\WHERe{\quad\text{where}\quad}
\nc\OR{\qquad\text{or}\qquad}
\nc\ANd{\quad\text{and}\quad}
\nc\anD{\ \ \ \text{and}\ \ \ }
\nc\ANDSIM{\qquad\text{and similarly}\qquad}
\nc{\COMMA}{,\qquad}
\nc{\COMMa}{,\quad}

\rnc\iff{\ \Leftrightarrow\ }
\nc\IFF{\qquad \Leftrightarrow\qquad }
\rnc\implies{\ \Rightarrow\ }
\nc\IFf{\quad\Leftrightarrow\quad}

\nc\pf{\begin{proof}}
\nc\epf{\end{proof}}
\nc\epfres{\hfill\qed}
\nc\epfreseq{\tag*{\qed}}

\let\oldproofname=\proofname
\renewcommand{\proofname}{\rm\bf{\oldproofname}}

\nc{\pfitem}[1]{\medskip\noindent #1.}
\nc{\firstpfitem}[1]{#1.}
\nc{\pfcase}[1]{\medskip\noindent {\bf Case #1.}}
\nc{\pfstep}[1]{\medskip\noindent {\bf Step #1.}}
\nc\aftercases{\medskip\noindent}
\nc{\pfclaim}[1]{\medskip\noindent{\bf Claim #1.}}
\nc{\pfclaimnn}{\medskip\noindent{\bf Claim.} } 
\nc\afterclaim{\medskip}
\nc{\pfsubcase}[1]{\medskip\noindent {\bf Subcase #1.}}

\nc\bu{{\bf u}}
\nc\bv{{\bf v}}
\nc\bw{{\bf w}}
\nc\ba{{\bf a}}
\nc\bb{{\bf b}}
\nc\bc{{\bf c}}
\nc\bzero{{\bf 0}}

\DeclareMathAlphabet{\mymathbb}{U}{BOONDOX-ds}{m}{n}
\nc\zero{\mymathbb 0}

\begin{document}

\title{\vspace{-1cm}Congruences of maximum regular subsemigroups of variants of finite full transformation semigroups\vspace{-1cm}}

\date{}
\author{}

\maketitle
\begin{center}
{\large 
Igor Dolinka,%
\hspace{-.25em}\footnote{Department of Mathematics and Informatics, University of Novi Sad, Trg Dositeja Obradovi\'ca 4, 21101 Novi Sad, Serbia. {\it Email:} {\tt dockie@dmi.uns.ac.rs}}
James East,%
\hspace{-.25em}\footnote{Centre for Research in Mathematics and Data Science, Western Sydney University, Locked Bag 1797, Penrith NSW 2751, Australia. {\it Email:} {\tt j.east@westernsydney.edu.au}}
Nik Ru\v{s}kuc%
\footnote{Mathematical Institute, School of Mathematics and Statistics, University of St Andrews, St Andrews, Fife KY16 9SS, UK. {\it Email:} {\tt nik.ruskuc@st-andrews.ac.uk}}
}
\end{center}

\vspace{0.5cm}

\begin{abstract}
\noindent
Let $\T_X$ be the full transformation monoid over a finite set $X$, and fix some $a\in\T_X$ of rank $r$.  The variant $\T_X^a$ has underlying set $\T_X$, and operation $f\star g=fag$.  We study the congruences of the subsemigroup $P=\Reg(\T_X^a)$ consisting of all regular elements of $\T_X^a$, and the lattice $\Cong(P)$ of all such congruences.  Our main structure theorem ultimately decomposes $\Cong(P)$ as a specific subdirect product of $\Cong(\T_r)$, and the full equivalence relation lattices of certain combinatorial systems of subsets and partitions.  We use this to give an explicit classification of the congruences themselves, and we also give a formula for the height of the lattice.

\medskip

\noindent
\emph{Keywords}: Congruence, congruence lattice, full transformation semigroup, variant, subdirect product.

\medskip

\noindent
MSC (2020): 20M20, 20M10, 08A30.

\end{abstract}

\tableofcontents

\section{Introduction}\label{sect:intro}

In the 1950s, Mal'cev classified the congruences of transformation monoids \cite{Malcev1952} and matrix monoids \cite{Malcev1953}.  These two papers initiated a new line of research in semigroup theory and were followed by a steady stream of papers, treating partial transformation monoids \cite{Sutov1961}, symmetric inverse monoids \cite{Liber1953} and many others.  More recent articles in this area have moved in other directions, including diagram monoids \cite{EMRT2018} and direct products of (linear) transformation monoids \cite{ABG2018}.

The paper \cite{ER2023} provides a unified framework for understanding the congruences  of many of the above monoids, as well as associated categories and their ideals; it also contains a fuller discussion of the history of the topic, and an extensive bibliography.  The monoids and categories amenable to analysis via the tools of \cite{ER2023} share a number of structural features: they are regular and stable; their ideals form a chain of order type $\leq\omega$; and they satisfy certain \emph{separation properties} related to Green's equivalences.  

The current article takes yet another direction in the congruence classification program, this time moving towards \emph{semigroup variants}.  The first studies of variants were by Hickey \cite{Hickey1983,Hickey1986}, building on older ideas of Lyapin \cite{Lyapin1960} and Magill \cite{Magill1967}, which were eventually unified categorically \cite{Sandwich1,Sandwich2}.  Given a semigroup $S$, and a fixed element $a\in S$, a new \emph{sandwich operation} ${\star}$ is defined by $x\star y = xay$ for $x,y\in S$.  This is associative, and the resulting semigroup $S^a=(S,{\star})$ is the \emph{variant} of $S$ with respect to $a$.  

The structure of a variant can be much more complex than that of the original semigroup.  
For example, consider the \emph{full transformation monoid}~$\T_4$, which consists of all self maps of $\{1,2,3,4\}$ under composition.  Figure \ref{fig:others} (left) shows the egg-box diagram of $\T_4$, while Figure \ref{fig:T4a} shows the variant~$\T_4^a$, where $a= \trans{1&2&3&4\\ 1&2&3&3}$.  (An egg-box diagram is a standard semigroup-theoretic visualisation tool; see for example \cite{CP1967}.)  
As these figures indicate, $\T_4$ has a chain of ideals, whereas $\T_4^a$ has an intricate ideal structure.

\newsavebox{\transa}
\savebox{\transa}{$\left(\begin{smallmatrix}1&2&3&4\\1&2&3&3\end{smallmatrix}\right)$}

\begin{figure}[t]
\begin{center}
\includegraphics[width=0.95\textwidth]{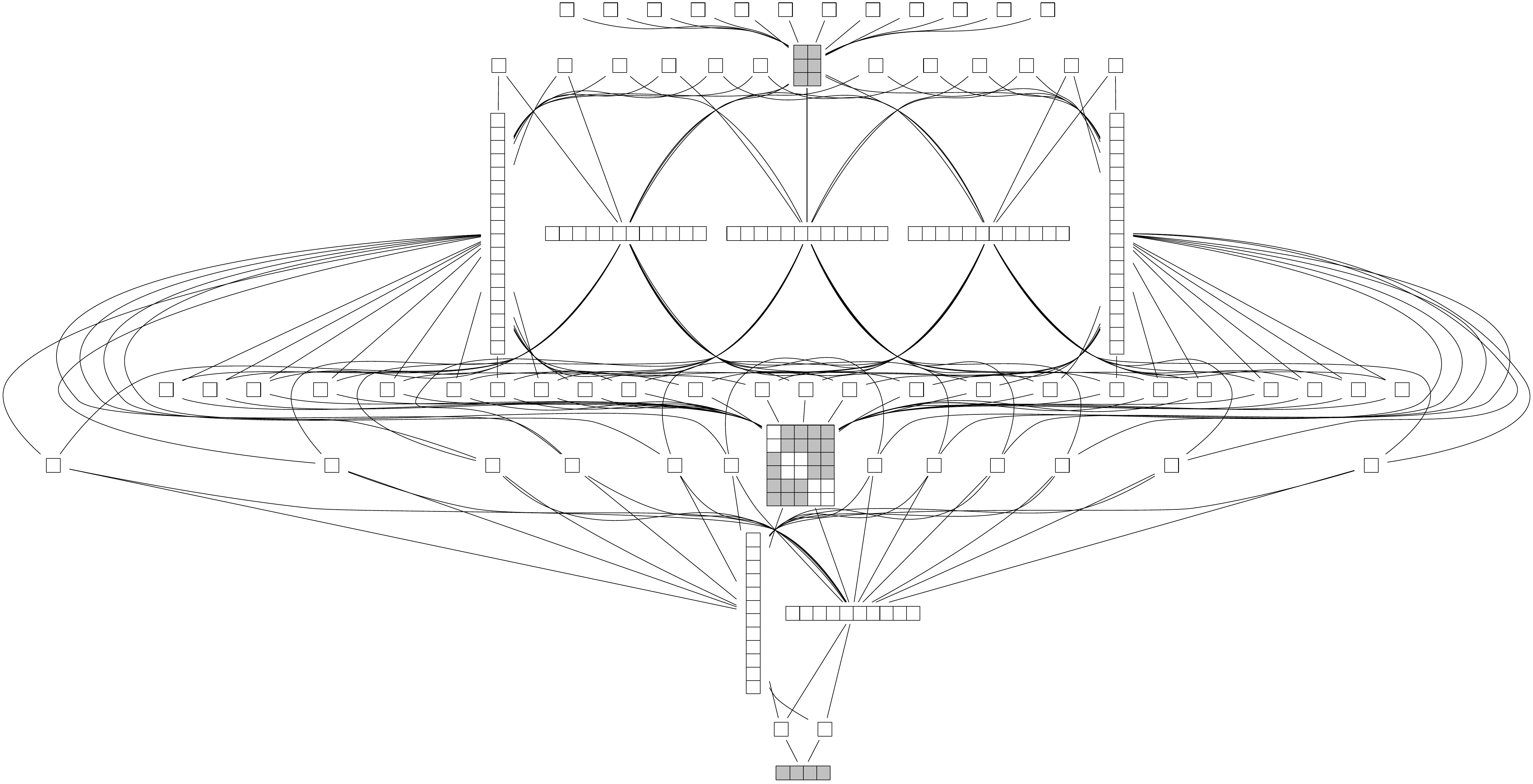}
\caption{Egg-box diagram of $\T_4^a$, where $a=\usebox{\transa}$.
}
\label{fig:T4a}
\end{center}
\end{figure}

\begin{figure}[t]
\begin{center}
\scalebox{0.9}{
\includegraphics[height=7.2cm]{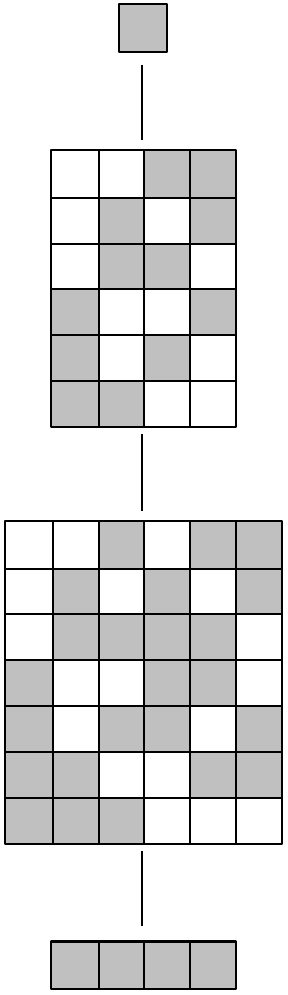}\qquad\qquad
\includegraphics[height=5cm]{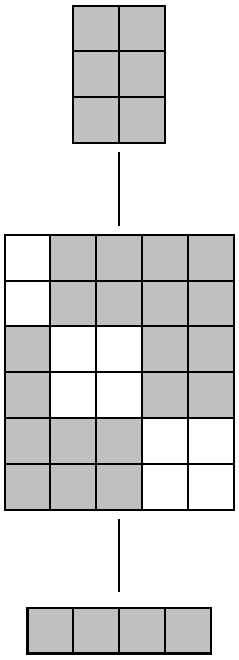}\qquad\qquad
\includegraphics[height=3.3cm]{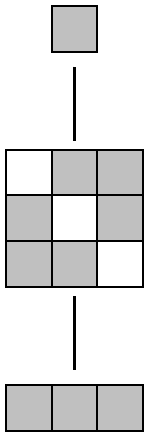}
}
\caption{Left to right: egg-box diagrams of $\T_4$, $\Reg(\T_4^a)$ and $\T_3$, where 
$a=\usebox{\transa}$.
}
\label{fig:others}
\end{center}
\end{figure}

Since~$\T_4$ is regular, it follows from \cite[Proposition 5]{KL2001} that the set $\Reg(\T_4^a)$ of regular elements of~$\T_4^a$ is a subsemigroup.  The egg-box diagram of $\Reg(\T_4^a)$ is shown in Figure \ref{fig:others} (middle), from which we can see that its ideals once again form a chain.  Less obvious, but still visible in the diagram, is that $\Reg(\T_4^a)$ is some kind of `inflation' of the (ordinary) full transformation semigroup~$\T_3$, which is pictured in Figure \ref{fig:others} (right).  Note that $\T_3$ appears here because the sandwich element $a\in\T_4$ has rank (image size) $3$.  This phenomenon was explored at length in the paper \cite{DE2015}, which systematically studied variants of finite full transformation semigroups and their regular subsemigroups.  The `inflation' was explained there in terms of certain `hat relations' extending Green's equivalences, and a natural surmorphism $\Reg(\T_4^a)\to\T_3$.

One of the striking consequences of Mal'cev's classification \cite{Malcev1952} is that the congruences of a finite full transformation monoid $\T_n$ form a chain.  
As explained in \cite{ER2023}, this is (very roughly speaking) a consequence of the normal subgroup lattices of the symmetric groups $\S_r$ ($1\leq r\leq n$) being chains, and $\T_n$ having the `separation properties' mentioned above.  
We will not need to introduce the latter here, because the variants $\Reg(\T_n^a)$ turn out not to satisfy them, and this route is not available to us.
But it is perhaps worth remarking that,
as shown in \cite{BEMMR2023}, the easiest way to verify the properties for $\T_n$  is to 
check that  its egg-box diagram has a certain combinatorial property, namely that distinct rows/columns in a non-minimal $\D$-class have distinct patterns of group and non-group $\H$-classes, as represented in egg-box diagrams by grey and white cells, respectively.  Examining Figure~\ref{fig:others}, one can check that this is indeed the case for~$\T_4$, but clearly not for~$\Reg(\T_4^a)$.

Since the general techniques developed in \cite{ER2023} do not apply to the finite regular variants $\Reg(\T_n^a)$, a new approach is required.  
Furthermore, there is no reason to expect that the congruence lattice~$\Cong(\Reg(\T_n^a))$ should be a chain,
and this can be verified computationally.  For example, the Semigroups package for GAP \cite{GAP,Semigroups} shows that the congruences of $\Reg(\T_4^a)$, with $a$ as above, form the lattice shown in Figure \ref{fig:lattice0}.  There are $271$ congruences, and the lattice is clearly not a chain; by contrast, the lattices $\Cong(\T_3)$ and $\Cong(\T_4)$ are chains of length~$7$ and~$11$, respectively.  Nevertheless, certain structural features of the lattice $\Cong(\Reg(\T_4^a))$ are visible in the diagram.  Indeed, the kernel~$\ka$ of the above surmorphism $\Reg(\T_4^a)\to\T_3$ corresponds in Figure~\ref{fig:lattice0} to the solid red vertex, and hence one can see the interval $[\De,\ka]$, as all vertices between it and  the solid blue vertex which represents the trivial congruence $\Delta$.  
There are a number of further intervals in $\Cong(\Reg(\T_4^a))$, isomorphic to subintervals of $[\Delta,\kappa]$, which are bounded by pairs of hollow red and blue vertices, and the entire lattice is a disjoint union of these intervals.

The preceeding observation is formalised in the first part of our main result, Theorem \ref{thm:main}\ref{it:main1}, which identifies the congruence lattice of a finite regular variant $\Reg(\T_X^a)$ as a specific subdirect product of $\Cong(\T_r)$ and~$[\De,\ka]$, where $r=\rank(a)$ and $\ka$ is the kernel of an analogous surmorphism ${\Reg(\T_X^a)\to\T_r}$.  The lattice $\Cong(\T_r)$ is well understood, thanks to Mal'cev \cite{Malcev1952}, and the remaining parts of Theorem~\ref{thm:main} describe the structure of the interval $[\De,\ka]$.  First, we have the direct product decomposition ${[\De,\ka] = [\De,\lam] \times [\De,\rho]}$, for certain congruences $\lam,\rho\sub\ka$ (Theorem \ref{thm:main}\ref{it:main2}).  Ultimately, the intervals $[\De,\lam]$ and $[\De,\rho]$ are shown to be subdirect products of families of full equivalence relation lattices over natural combinatorial systems of subsets and partitions (Theorem \ref{thm:main}\ref{it:main3} and \ref{it:main4}).  

The paper is organised as follows.  After giving preliminaries in Section \ref{sect:prelim}, we state our main result in Section \ref{sec:statement}.  We then pause to record some auxiliary lemmas in Section \ref{sec:auxi}, before giving the proofs of the various parts of the main result in Sections \ref{sect:join}--\ref{sec:RL}.  The information gathered during this process will then be combined in Section \ref{sect:class} to give a classification of the congruences themselves.  As an application of our structure theorem, we give a formula for the height of the congruence lattice in Section \ref{sect:height}.  The paper concludes in Section \ref{sect:conc} with a discussion of directions for future work.

\subsection*{Acknowledgements}

This work is supported by the following grants:
F-121 of the Serbian Academy of Sciences and Arts;
Future Fellowship FT190100632 of the Australian Research Council;
EP/S020616/1 and EP/V003224/1 of the Engineering and Physical Sciences Research Council.
The first author is also partially supported by the Ministry of Science, Technological Development, and Innovations of the Republic of Serbia.

\begin{figure}[t]
\begin{center}
\scalebox{0.8}{
\begin{tikzpicture}
\node[above right]  () at (0,0) {\includegraphics[width=\textwidth]{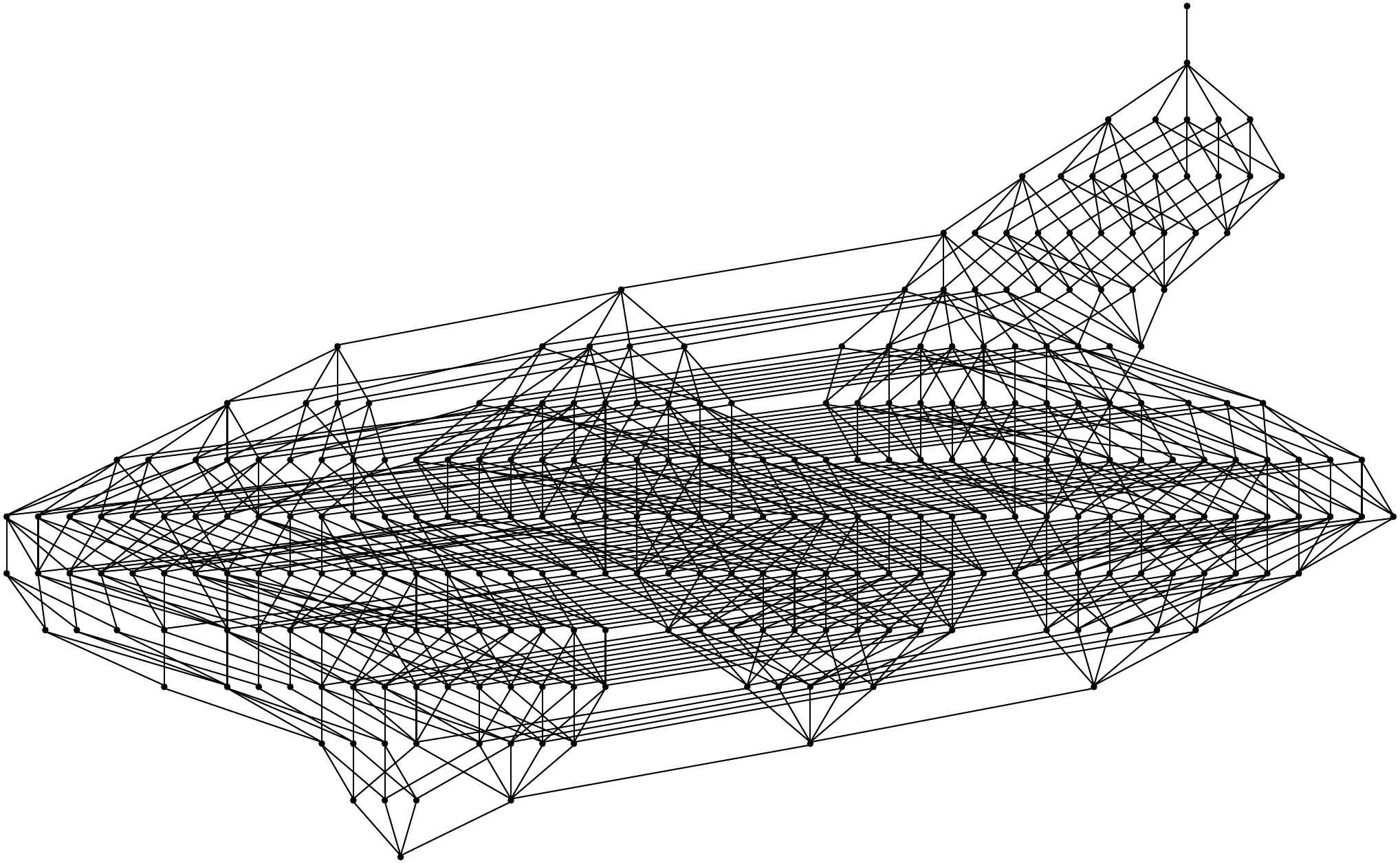}};
\foreach \x/\y in {
9.98/1.6,
13.42/2.29,
14/6.4,
14.28/7.1,
15.04/7.8
} {\fill [blue] (\x,\y)circle(.1); \fill [white] (\x,\y)circle(.05);}
\foreach \x/\y in {
5/0.22
} {\fill [blue] (\x,\y)circle(.1);}
\foreach \x/\y in {
7.68/7.1,
11.6/7.79,
12.55/8.49,
13.6/9.16,
14.56/9.85,
14.56/10.54
} {\fill [red] (\x,\y)circle(.1); \fill [white] (\x,\y)circle(.05);}
\foreach \x/\y in {
4.23/6.41
} {\fill [red] (\x,\y)circle(.1);}
\end{tikzpicture}
}
\caption{The congruence lattice of $\Reg(\T_4^a)$, where $a=\usebox{\transa}$; cf.~Figures \ref{fig:lattice} and \ref{fig:lattice1}.}
\label{fig:lattice0}
\end{center}
\end{figure}

\section{Preliminaries}\label{sect:prelim}

In this section we establish notation and gather some basic background facts concerning semigroups.  Unless otherwise indicated, proofs of the various assertions can be found in a standard text such as \cite{CP1967} or \cite{Howie1995}.  We also review some results concerning congruences from \cite{EMRT2018,ER2023} and variants of finite full transformation semigroups from \cite{DE2015}; see also \cite[Chapter 13]{GMbook}.

\subsection{Green's relations}

Let $S$ be a semigroup.  We write $S^1$ for the \emph{monoid completion} of $S$.  Specifically, $S^1=S$ if $S$ happens to be a monoid; otherwise $S^1=S\cup\{1\}$, where $1$ is a symbol not belonging to $S$, acting as an adjoined identity element.  Define preorders $\leqL$, $\leqR$ and $\leqJ$, for $x,y\in S$, by
\[
x\leqL y \iff x\in S^1y \COMMA x\leqR y\iff x\in yS^1 \AND x\J y \iff x\in S^1yS^1.
\]
These induce equivalences ${\L} = {\leqL}\cap{\geqL}$, ${\R} = {\leqR}\cap{\geqR}$ and ${\J} = {\leqJ}\cap{\geqJ}$.  Note that ${x\L y \iff S^1x=S^1y}$, with similar statements holding for $\R$ and $\J$.  We also have the equivalences ${\H}={\L}\cap{\R}$ and ${\D}={\L}\vee{\R}$, where the latter denotes the join of $\L$ and $\R$ in the lattice~$\Eq(S)$ of all equivalences on $S$, i.e.~$\D$ is the transitive closure of the union ${\L}\cup{\R}$.   It turns out that in fact ${\D}={\L}\circ{\R}={\R}\circ{\L}$.  If $S$ is finite, then ${\D}={\J}$.  The $\H$-class of any idempotent is a group; all group $\H$-classes contained in a common $\D$-class are isomorphic.

If $\K$ denotes any of $\L$, $\R$, $\J$, $\H$ or $\D$, we denote by $K_x$ the $\K$-class of $x$ in $S$.  The set $S/{\J}=\set{J_x}{x\in S}$ of all $\J$-classes is partially ordered by
\begin{equation}\label{eq:leqK}
J_x \leq J_y \iff x\leqJ y \qquad\text{for $x,y\in S$.}
\end{equation}
The above relations are collectively referred to as \emph{Green's relations}, and were introduced in \cite{Green1951}.  The next two results are well known, and appear for example as 
Lemma 2.2.1 and Proposition 2.3.7 in \cite{Howie1995}.

\begin{lemma}[Green's Lemma]\label{lem:GL}
Let $x$ and $y$ be $\R$-related elements of a semigroup $S$, so that $y=xs$ and $x=yt$ for some $s,t\in S^1$.  Then the maps
\[
L_x\to L_y:u\mt us \AND L_y\to L_x:v\mt vt
\]
are mutually inverse $\R$-preserving bijections.  Moreover, these restrict to mutually inverse bijections
\[
H_x\to H_y:u\mt us \AND H_y\to H_x:v\mt vt.  \epfreseq
\]
\end{lemma}

Lemma \ref{lem:GL} has a left-right dual, which will also be referred to as Green's Lemma.

\begin{lemma}\label{lem:237}
If $x$ and $y$ are elements of a semigroup $S$, then $xy\in R_x\cap L_y$ if and only if $L_x\cap R_y$ contains an idempotent.  \epfres
\end{lemma}

An element $x\in S$ is \emph{regular} if $x\in xSx$.  We denote by $\Reg(S)$ the set of all regular elements, but note that this need not be a subsemigroup.  If $x$ is regular, then $D_x\sub\Reg(S)$.  We say that $S$ itself is \emph{regular} if $S=\Reg(S)$.  If $x$ is regular, then there exist idempotents $e,f\in E(S)$ with $e\L x\R f$, and we then have $xe=x=fx$.  Here as usual $E(S) = \set{e\in S}{e=e^2}$ denotes the set of all idempotents of $S$.

A $\J$-class $J$ of $S$ is \emph{stable} if
\[
x\J ax \implies x\L ax \AND x\J xa\implies x\R xa \qquad\text{for all $x\in J$ and $a\in S$.}
\]
A stable $\J$-class is in fact a $\D$-class \cite[Proposition 2.3.9]{Lallement1979}.  We say that $S$ itself is \emph{stable} if every $\J$-class is stable.  All finite semigroups are stable \cite[Theorem A.2.4]{RS2009}.

\subsection{Congruences}

An equivalence $\si$ on a semigroup $S$ is a \emph{left congruence} if it is \emph{left compatible}, meaning that
\[
(x,y)\in\si \implies (ax,ay)\in\si \qquad\text{for all $a,x,y\in S$.}
\]
\emph{Right compatibility} and \emph{right congruences} are defined dually.  
Note for example that $\L$ is a right congruence, and $\R$ a left congruence.  
An equivalence $\sigma$ on $S$ is a \emph{congruence} if it is both left and right compatible, which is equivalent to $\sigma$ satisfying
\[
(a,b),(x,y)\in\si \implies (ax,by)\in\si \qquad\text{for all $a,b,x,y\in S$.}
\]
The set $\Cong(S)$ of all congruences of $S$ is a lattice under inclusion, called the \emph{congruence lattice} of~$S$, and is a sublattice of $\Eq(S)$.  In particular, the meet and join of congruences $\si,\tau\in\Cong(S)$ are the same as in $\Eq(S)$, so $\si\wedge\tau=\si\cap\tau$, while $\si\vee\tau$ is the least equivalence containing $\si\cup\tau$.  The bottom and top elements of $\Cong(S)$ are the trivial and universal relations:
\[
\De_S = \set{(x,x)}{x\in X} \AND \nab_S = S\times S.
\]

A (possibly empty) subset $I\sub S$ is an \emph{ideal} if $SI\cup IS\sub I$.  Any such $I$ determines the \emph{Rees congruence}
\[
R_I = \nab_I \cup \De_S = \set{(x,y)\in S\times S}{x=y \text{ or } x,y\in I}.
\]
Note that $R_\es=\De_S$ and $R_S=\nab_S$.

Ideals can be combined with group $\H$-classes to create another family of congruences as follows.  Let $I$ be an ideal of $S$.  As $I$ is a union of $\J$-classes, so too is $S\sm I$.  Suppose $J$ is a $\J$-class that is minimal in the poset $(S\sm I)/{\J}$ under the $\leq$ order defined in \eqref{eq:leqK}.  Suppose also that $J$ is regular and stable, so that in fact $J$ is a $\D$-class.  Let $G$ be a group $\H$-class contained in $J$, and let $N\normal G$ be a normal subgroup.  The relation
\[
\nu_N = S^1(N\times N)S^1 \cap (J\times J) = \bigset{(axb,ayb)}{x,y\in N,\ a,b\in S^1,\ axb,ayb\in J}
\]
is an equivalence on $J$, and $\nu_N\sub{\H}$ \cite[Lemma 3.17]{EMRT2018}.  Moreover, the relation
\begin{equation}\label{eq:RIN}
R_{I,N} = \nab_I \cup \nu_N \cup \De_S
\end{equation}
is a congruence of $S$ \cite[Proposition 3.23]{EMRT2018}.  
As explained in \cite[Remark 2.11]{ER2023}, the set of congruences $\set{R_{I,N}}{N\normal G}$ forms a sublattice of $\Cong(S)$ isomorphic to the normal subgroup lattice of $G$.  In the case that $N=\{1\}$ is the trivial (normal) subgroup, $R_{I,N}$ is just the Rees congruence $R_I$.
It was shown in \cite[Lemma 3.16]{EMRT2018} that the $\nu_N$ relations are independent of the choice of group $\H$-class, in the sense that for any two such groups $G_1,G_2\sub J$, and for any normal subgroup $N_1\normal G_1$, we have $\nu_{N_1}=\nu_{N_2}$ for some $N_2\normal G_2$.

\begin{lemma}\label{lem:28}
Let $D$ be a stable regular $\J$-class of a semigroup $S$ (so that $D$ is in fact a $\D$-class), and let $\si\in\Cong(S)$.  Fix a group $\H$-class $G\sub D$, and let $e$ be the identity of $G$.  Then
\[
\si \cap {\H}\restr_D = \nu_N \WHERE N = \set{g\in G}{e\mr\si g}.
\]
\end{lemma}

\pf
Let $I$ be the union of all the $\J$-classes below $D$, so that $I$ is a (possibly empty) ideal of $S$, and consider the congruence $\tau = \si\cap R_{I,G}\in\Cong(S)$.  Since
\[
R_{I,G} = \nab_I \cup \nu_G \cup \De_S = \nab_I \cup {\H}\restr_D \cup \De_S,
\]
we have
\begin{equation}\label{eq:tau}
\tau = \si\restr_I \cup (\si\cap{\H}\restr_D) \cup \De_S.
\end{equation}
In particular, $\tau\restr_D = \si\cap{\H}\restr_D \sub {\H}$, so it follows from \cite[Lemma 2.8]{ER2023} that
\[
\tau\restr_D = \nu_{N'}, \WHERE N' = \set{g\in G}{e\mr\tau g}.
\]
We have already observed that $\tau\restr_D = \si\cap{\H}\restr_D$, so it follows that $\si\cap{\H}\restr_D = \nu_{N'}$.  To complete the proof, it remains to show that $N'=N$, i.e.~that $(e,g)\in\tau \iff (e,g)\in\si$ for all $g\in G$.  But for any such $g$, we have
\begin{align*}
(e,g)\in\tau &\iff (e,g)\in\si\cap{\H}\restr_D &&\text{by \eqref{eq:tau}, as $e,g\in D$}\\
&\iff (e,g)\in\si &&\text{since $(e,g)\in{\H}$ (as $e,g\in G$).}  \qedhere
\end{align*}
\epf

\subsection{Finite transformation semigroups and their congruences}

Fix a finite set $X$ of size $n$, and let $\T_X$ be the \emph{full transformation semigroup} over $X$, i.e.~the semigroup of all mappings $X\to X$ under composition.  Green's preorders on $\T_X$ are determined by images, kernels and ranks.  These parameters are defined, for $f\in\T_X$, by
\[
\im(f) = \set{xf}{x\in X},\quad
\ker(f) = \set{(x,y)\in X\times X}{xf=yf} \ANd
\rank(f) = |{\im(f)}| = |X/\ker(f)|.
\]
For $f,g\in\T_X$ we have
\begin{align*}
f\leqL g&\iff\im(f)\sub\im(g),& f\leqR g &\iff \ker(f)\supseteq\ker(g),& f\leqJ g&\iff \rank(f)\leq\rank(g),\\
f\L g&\iff\im(f)=\im(g),& f\R g &\iff \ker(f)=\ker(g),& f\J g&\iff \rank(f)=\rank(g).
\end{align*}
The ${\J}={\D}$-classes and non-empty ideals of $\T_X$ are the sets
\[
D_r=\set{f\in\T_X}{\rank(f)=r} \AND I_r=\set{f\in\T_X}{\rank(f)\leq r} \qquad\text{for $1\leq r\leq n$,}
\]
and they form chains, $D_1<\cdots<D_n$ and $I_1\subset\cdots\subset I_n=\T_n$.  The top $\D$-class $D_n$ is equal to the symmetric group~$\S_X$.  Group $\H$-classes in $D_r$ are isomorphic to $\S_r$.  In particular, we can identify $\S_r$ with any such group $\H$-class, and we can then speak of the congruences $R_{I_{r-1},N}$, for any $1\leq r\leq n$, and any $N\normal\S_r$.  For $r=1$ we interpret $I_0=\es$; as $\S_1$ is the trivial group, the only such congruence arising for $r=1$ is $R_{I_0,\S_1} = \De_{\T_X}$.  The following major result by Mal'cev initiated research into congruence lattices of important concrete semigroups, and is one of the key ingredients on which the present work is built:

\begin{thm}[Mal'cev \cite{Malcev1952}]\label{thm:Malcev}
If $X$ is a finite set of size $n$, then the congruence lattice of $\T_X$ is a chain:
\[
\Cong(\T_X) = \{\nab_{\T_X}\} \cup \set{R_{I_{r-1},N}}{1\leq r\leq n,\ N\normal\S_r}.  \epfreseq
\]
\end{thm}

\subsection{Variants of finite transformation semigroups}\label{subsect:TXa}

Again we fix a finite set $X$.  We also fix a transformation $a\in\T_X$, and let ${\star}$ be the \emph{sandwich operation} on $\T_X$ defined by
\[
f\star g = fag \qquad\text{for $f,g\in\T_X$.}
\]
Then $\T_X^a = (\T_X,{\star})$ is the \emph{variant} of $S$ with respect to $a$.  Since there exists a permutation $p\in\S_X$ such that $ap$ is an idempotent, and since the map
\begin{equation}\label{eq:TXap}
\T_X^a\to\T_X^{ap}:f\mt p^{-1}f
\end{equation}
is an isomorphism, we may assume without loss of generality that $a$ is itself an idempotent.  We will adopt this set-up throughout the paper. Any statement that is made with that assumption can readily be translated into the case of an arbitrary sandwich element using the isomorphism \eqref{eq:TXap}.  Using standard tabular notation for transformations, we will write
\begin{equation}\label{eq:a}
a = \trans{A_1&\cdots&A_r\\a_1&\cdots&a_r} , \quad\text{so}\quad \rank(a) = r \COMMa \im(a) = \{a_1,\ldots,a_r\} \ANd X/\ker(a) = \{A_1,\ldots,A_r\}. 
\end{equation}
Since $a$ is an idempotent, we have $a_i\in A_i$ for all $i$.  We now outline some results concerning $\T_X^a$ and certain important subsemigroups; proofs of the assertions can be found in \cite{DE2015,Sandwich1,Sandwich2}.

The set of all regular elements of $\T_X^a$ is a subsemigroup, and we denote it by
\[
P = \Reg(\T_X^a).
\]
Many characterisations of $P$ were given in \cite{DE2015}, the most useful for our purposes being
\[
P = \set{f\in\T_X}{\rank(afa)=\rank(f)}.
\]
Since $\rank(afa)\leq\rank(a)=r$ for any $f\in\T_X$, the elements of $P$ have rank at most $r$.  The set
\[
T = a\T_Xa = a\star\T_X\star a = \set{afa}{f\in\T_X}
\]
is a subsemigroup of both $\T_X$ and $\T_X^a$ ($fg=f\star g$ for $f,g\in T$), and $T\cong\T_r$, where we recall that $r=\rank(a)$.  Since $afa=f$ for all $f\in T$, certainly $\rank(afa)=\rank(f)$ for all such $f$, and so $T\sub P$.  
It also follows that $T=aPa=a\star P\star a$, and that we have a retraction
\begin{equation}\label{eq:phi}
\phi:P\to T:f\mt \ol f = afa.
\end{equation}
That is, $\ol{f\star g}=\ol f\star \ol g(=\ol f\ol g)$ for all $f,g\in P$, and $\ol h=h$ for all $h\in T$.

Our results and proofs often involve the interplay between $P$ and $T$ via~$\phi$.
We will distinguish between standard semigroup theoretic notation as applied to $P$ or $T$ by using appropriate superscripts.
Thus, for example, if $\K$ is any of Green's equivalences $\L$, $\R$, $\J$, $\H$ or~$\D$, we will write $\K^P$ and~$\K^T$ for $\K$ on $P$ and $T$, respectively.
These relations have exactly the same characterisation as in~$\T_X$.  Namely, if $S$ is either $P$ or $T$, then for any $f,g\in S$ we have
\[
f\L^S g \iff \im(f) = \im(g) \COMMa f \R^S g \iff \ker(f) = \ker(g) \ANd f \D^S g \iff \rank(f) = \rank(g).
\]
Regarding the $\K$-classes, we have $K_f^T\sub K_f^P$ for any $f\in T(\sub P)$, and this inclusion is typically strict.  We do note, however, that $H_f^T=H_f^P$ for any $f\in T$, although $P$ typically has more $\H$-classes than~$T$.
The ${\D^S}(={\J^S})$-classes and non-empty ideals of $S$ (still denoting either $P$ or $T$) are the sets
\[
D_q^S = \set{f\in S}{\rank(f)=q} \AND I_q^S = \set{f\in S}{\rank(f)\leq q} \qquad\text{for $1\leq q\leq r$.}
\]
These are ordered by $D_1^S<\cdots<D_r^S$ and $I_1^S\subset\cdots\subset I_r^S = S$.  We also define~$I_0^S=\es$.

An important role will be played by 
the preimages under $\phi$ of Green's relations on $T$:
\[
{\wh \K^P} = {\K^T}\phi^{-1} = \set{(f,g)\in P\times P}{(\ol f,\ol g)\in{\K^T}}.
\]
We write $\wh K_f^P$ for the $\wh\K^P$-class of $f$ in~$P$.  
Note that $\hL^P$ is a right congruence of~$P$, being a pre-image of the right congruence $\L^T$; likewise, $\hR^P$ is a left congruence.  It follows from \cite[Lemma 3.11]{Sandwich1} that ${\K^P}\sub{\hK^P}\sub{\hD^P}={\D^P}$ for ${\K}={\L}$, $\R$ or $\H$, and of course then ${\hJ^P}={\J^P}={\D^P}={\hD^P}$, as $P$ is finite.  The next result gathers some facts from \cite[Theorem 3.14]{Sandwich1}.  For the statement, recall that a \emph{rectangular band} is (isomorphic to) a semigroup of the form $L\times R$ with product $(l_1,r_1)(l_2,r_2)=(l_1,r_2)$, and a \emph{rectangular group} is (isomorphic to) a direct product of a rectangular band and a group.

\begin{lemma}\label{lem:DE}
Let $f\in P$.
\ben
\item \label{DE1} The restriction $\phi\restr_{H_f^P}$ is a bijection $H_f^P\to H_{\ol f}^T$.
\item \label{DE2} If $H_{\ol f}^T$ is a group, then $\wh H_f^P$ is a rectangular group, and in particular a union of group $\H^P$-classes.
\item \label{DE3} If $H_{\ol f}^T$ is not a group, then $\wh H_f^P$ is a union of non-group $\H^P$-classes.  \epfres
\een
\end{lemma}

\subsection{Direct and subdirect products}\label{sec:sdps}

Our main result will describe the congruence lattice $\Cong(P)$ as successively decomposed into direct and subdirect products of smaller lattices. Here we introduce terminology and notation for these products. 
What follows is presented in the context of lattices, but is in fact completely general and applies to any algebraic structures.

Let $L_i$ ($i\in I$) be a collection of lattices. The \emph{direct product} $\prod_{i\in I} L_i$ is the lattice with underlying set consisting of all $I$-tuples $(a_i)_{i\in I}=(a_i)$, with each $a_i\in L_i$, and with component-wise meet and join operations
$(a_i)\wedge (b_i)=(a_i\wedge b_i)$ and $(a_i)\vee (b_i)=(a_i\vee b_i)$.
For every $j\in I$, the projection $\pi_j:\prod_{i\in I} L_i\rightarrow L_j$ is a lattice surmorphism. A sublattice $L\leq \prod_{i\in I} L_i$ is said to be a \emph{subdirect product} of the $L_i$ if $\pi_i(L)=L_i$ for all $i\in I$.
A \emph{subdirect embedding} is an injective morphism $L\rightarrow \prod_{i\in I} L_i$ whose image is a subdirect product.

We now  review the well-known criteria for the existence of a direct decomposition in the case of two factors, or a subdirect decomposition for any number of factors. 

\begin{prop}[see {\cite[Theorem II.7.5]{BS1981}}]\label{pr:dp}
Let $L$ be a lattice, and let $\phi_i: L\rightarrow L_i$ $(i=1,2)$ be two lattice surmorphisms.
If $\ker(\phi_1)\cap \ker(\phi_2)=\Delta_L$  and  $\ker(\phi_1)\circ\ker(\phi_2)=\nabla_L$ then
$L\cong L_1\times L_2$ via $a\mapsto (\phi_1(a),\phi_2(a))$.  \epfres
\end{prop}

\begin{prop}[see {\cite[Lemma II.8.2]{BS1981}}]\label{pr:sdp}
Let $L$ be a lattice, and let $\phi_i: L\rightarrow L_i$ $(i\in I)$ be lattice surmorphisms.
If $\bigcap_{i\in I}\ker(\phi_i)=\Delta_L$  then the mapping
$a\mapsto (\phi_i(a))$ is a subdirect embedding of $L$ into $\prod_{i\in I} L_i$.  \epfres
\end{prop}

\section{The statement of the main result}
\label{sec:statement}

The main result of this paper is a detailed structural description of the congruence lattice of the regular part of a variant of a finite full transformation monoid. This description is in several stages. The purpose of this section is to give full statements for each of the stages, and, in the process, fix the concepts and notation that will be used subsequently.  
To begin with:
\begin{itemize}
\item
$P$ will denote the semigroup $\Reg(\T_X^a)$, i.e.~the regular part of the variant $\T_X^a$ of the full transformation monoid on a finite set $X$, with respect to the sandwich element $a=\trans{A_1&\cdots&A_r\\a_1&\cdots&a_r}\in\T_X$, which we assume is an idempotent.
\item
$\phi$ is the retraction $f\mapsto \ol f= afa$ from~\eqref{eq:phi}, $T\cong \T_r$ is its image, and $\kappa$ its kernel:
\[
\ka = \ker(\phi) = \set{(f,g)\in P\times P}{\ol f=\ol g}.
\]
\item
We additionally define $\lambda=\kappa\cap{\L^P}$ and $\rho=\kappa\cap{\R^P}$.  
\item For congruences $\xi\in\Cong(\T_r)$ and $\theta\in [\Delta_P,\kappa](\sub\Cong(P))$, we define
\begin{equation}\label{eq:rank}
\rk(\xi) = \max\set{q}{R_{I_q^{\T_r}}\sub\xi}\AND \rk(\th) = \max\set{q}{\ka\cap R_{I_q^P} \sub\th}.
\end{equation}

\end{itemize}
At this point, we have sufficient notation for the first two parts of our main theorem. For the remaining two parts we need to do a bit more work.

For a positive integer $k$ we write $[k]=\{1,\ldots,k\}$.  Consider a non-empty subset ${\es\not=I\subseteq [r]}$, where we recall that $r=\rank(a)$.
Let $\cC_I$ be the set of all cross-sections of $\set{A_i}{ i\in I}$; thus
an element of~$\cC_I$ has the form $C=\set{ c_i}{ i\in I}$ with each $c_i\in A_i$.
For $\emptyset\neq J\subseteq I\sub[r]$ and $C\in\cC_I$ as above, define $C\restr_J=\set{ c_j}{ j\in J}$.
For a relation $\psi$ on $\cC_I$ define 
$\psi\restr_J=\bigset{ (C\restr_J,C'\restr_J)}{ (C,C')\in \psi}$, which is a relation on $\cC_J$.

A \emph{partition} of a set $Z$ is a set of non-empty, pairwise-disjoint subsets of $Z$ (called blocks), which cover $Z$.
If the blocks of a partition $\bI$ are all contained in blocks of another partition $\bJ$ (of the same set), we say that $\bI$ \emph{refines} $\bJ$, and write $\bJ\preceq \bI$. 
For a positive integer $k$, we denote the \emph{trivial partition} of $[k]$ by $\ldb k\rdb = \bigset{\{i\}}{ i\in [k]}$.
Clearly, $\bI\preceq \ldb k\rdb$ for every partition $\bI$ of $[k]$.

Now consider a  partition $\bI$ of $[r] $.
Let $\cP_\bI$ be the set of all partitions of $[n]$ of the form $\bP=\set{ P_I}{ I\in \bI}$
such that $P_I\cap \im(a)=\set{ a_i}{ i\in I}$ for each $I\in\bI$.
For partitions $\bJ\preceq \bI\preceq\ldb r\rdb$, and $\bP\in\cP_\bI$ as above, define $\bP\restr_\bJ\in\cP_\bJ$ to be the partition $\set{ Q_J}{ J\in\bJ}$,
where $Q_J=\bigcup \set{ P_I}{ I\in\bI,\ I\subseteq J}$ for each $J\in\bJ$.
For a relation $\psi$ on $\cP_\bI$ define
$\psi\restr_\bJ=\bigset{ (\bP\restr_\bJ,\bP'\restr_\bJ)}{(\bP,\bP')\in\psi}$, which is a relation on $\cP_\bJ$.

Here then is our main result.  As per our convention introduced in Subsection \ref{subsect:TXa}, it is formulated in terms of an idempotent sandwich element.  There is no loss of generality, due to the isomorphism~\eqref{eq:TXap}.

\begin{thm}
\label{thm:main}
Let $X$ be a finite set, let $a\in\T_X$ be an idempotent of rank $r\geq2$, and let $P=\Reg(\T_X^a)$.
\ben
\item\label{it:main1}
The lattice $\Cong(P)$  subdirectly embeds into $\Cong(\T_r)\times [\Delta_P,\kappa]$, with image
\[
\bigset{ (\xi,\theta)\in\Cong(\T_r)\times [\Delta_P,\kappa]}{ \rank(\xi)\leq\rank(\theta)}.
\]
\item\label{it:main2}
The interval $[\Delta_P,\kappa]$ is isomorphic to the direct product $[\Delta_P,\lambda]\times [\Delta_P,\rho]$.
\item\label{it:main3}
The interval $[\Delta_P,\rho]$ subdirectly embeds into the direct product 
$\prod_{\emptyset\neq I\subseteq [r]} \Eq(\cC_I)$ of full lattices of equivalence relations on the sets $\cC_I$, with image
\[
\Bigset{ (\psi_I) \in \prod_{\emptyset\neq I\subseteq [r]} \Eq(\cC_I)}{ \psi_I\restr_J\subseteq \psi_J\text{ for all } \emptyset\neq J\subseteq I\subseteq [r]}.
\]
\item\label{it:main4}
The interval $[\Delta_P,\lambda]$ subdirectly embeds into the direct product 
$\prod_{\bI\preceq\ldb r\rdb} \Eq(\cP_\bI)$ of full lattices of equivalence relations on the sets $\cP_\bI$, with image
\[
\Bigset{ (\psi_\bI) \in\prod_{\bI\preceq \ldb r\rdb} \Eq(\cP_\bI)}{ \psi_\bI\restr_\bJ\subseteq \psi_\bJ
\text{ for all } \bJ\preceq \bI\preceq \ldb r\rdb}.
\]
\een
\end{thm}

\begin{rem}\label{rem:1}
The $r=1$ case was excluded from Theorem \ref{thm:main}, but it is easy to understand the semigroup $P=\Reg(\T_X^a)$ and its congruence lattice in this case.  Indeed, here $P=D_1$ is a right zero semigroup, and hence every equivalence is a congruence, meaning that $\Cong(P)=\Eq(P)$.  We also have $T=\{a\}$, and so $\ka=\nab_P$, $\lam={\L^P}=\De_P$ and $\rho={\R^P}=\nab_P$.  Parts \ref{it:main2}--\ref{it:main4} of the theorem are then trivial.  Regarding part \ref{it:main1}, $\Cong(P)$ is of course isomorphic to $\Cong(\T_1)\times[\De_P,\ka]$, as~$\T_1$ is trivial and $[\De_P,\ka]=[\De_P,\nab_P]=\Cong(P)$.  However, there is a slight discrepancy in the stated image of the embedding when $r=1$, as the unique congruence of $\T_1$ (i.e.~$\De_{\T_1}=\nab_{\T_1}$) has rank $1$, while the non-universal congruences in $[\De_P,\ka](=\Cong(P))$ have rank $0$.  This `problem' could be fixed by introducing the convention that $\rank(\De_{\T_1})=0$.
\end{rem}

The four parts of Theorem \ref{thm:main} will be proved as Theorems \ref{thm:Psi}, \ref{th:dk}, \ref{thm:DeR} and \ref{thm:DeL}.  In fact, each of these results provides additional information, in the form of an explicit expression for the (sub)direct embedding in question.  En route to proving them, we will gather enough information to deduce an explicit classification of the congruences of $P$, which will be given in Theorem \ref{thm:class}.

\section{Auxiliary results}
\label{sec:auxi}

The proofs of the four parts of Theorem \ref{thm:main} will be given in Sections \ref{sect:join}--\ref{sec:RL}.  To keep those sections focussed on their main objectives, this section gathers some technical observations that will be subsequently used. 
They concern the relationship between congruences on $P$ and on $T$ (Subsection \ref{sub:restr}), as well as a couple of technical properties of congruences containing $\hH^P$-related pairs (Subsection \ref{subsect:hHP}).

\subsection{Restrictions and lifts}
\label{sub:restr}

Given that $T$ is both a subsemigroup of $P$ and a homomorphic image (via $\phi$) of $P$, any congruence $\si\in\Cong(P)$ induces the restriction to $T$ and the image in $T$:
\[
\si\restr_T = \si\cap(T\times T) = \set{(f,g)\in\si}{f,g\in T} \AND
\overline{\sigma}=\set{(\overline{f},\overline{g})}{(f,g)\in \sigma}.
\]
Conversely, given a congruence $\xi$ on $T$, we can `lift' it to the congruence $\xi^\sharp$ on $P$ generated by $\xi$.
The next lemma establishes some important connections between these constructions.  The first part will be used frequently without explicit reference.

\begin{lemma}\label{la:reli}
\ben
\item\label{it:reli1}
For any $\si\in\Cong(P)$, we have $\si\restr_T = \overline{\sigma}$.
\item\label{it:reli2}
For any $\xi\in \Cong(T)$, we have $\xi = \xi^\sharp\restr_T $.
\een
\end{lemma}

\pf
\firstpfitem{\ref{it:reli1}}
If $(f,g)\in\si\restr_T$,  then $(f,g)\in\si$ and $f,g\in T$, and  hence $(f,g)=(\ol f,\ol g)\in\ol\si$. Thus $\sigma\restr_T\subseteq \overline{\sigma}$.
For the reverse inclusion, consider $(\overline{f},\overline{g})\in\overline{\sigma}$, with $(f,g)\in\sigma$.
Then certainly $\overline{f},\overline{g}\in T$, and 
we also have $(\ol f,\ol g) = (afa,aga) = (a\star f\star a,a\star g\star a)\in\si$, so $(\ol f,\ol g)\in\si\restr_T$.

\pfitem{\ref{it:reli2}}
This can be proved by noting that $T = a\star P\star a$ is a local monoid of $P$, and that local monoids have the congruence extension property.  Alternatively, it also follows from Lemma \ref{lem:RN} below.
\epf

Further to part \ref{it:reli2} of the previous lemma, at times we will actually need to know the exact form that a congruence $\xi^\sharp$ takes,
depending on the form of $\xi$ as specified in Theorem~\ref{thm:Malcev}. We now establish the notation for doing this.

For any $1\leq q\leq r$, group $\H$-classes in the $\D$-classes $D_q^P(\sub P)$ and $D_q^T(\sub T)$ are isomorphic to~$\S_q$.  Thus, each normal subgroup $N\normal\S_q$ gives rise to a congruence on both $P$ and $T$, as in \eqref{eq:RIN}.  We compress the notation for these congruences as follows:
\[
R_N^S = R_{I_{q-1}^S,N} = \nab_{I_{q-1}^S} \cup \nu_N^S \cup \De_S \qquad\text{where $S$ stands for either $P$ or $T$.}
\]
Since $T\leq P$, we have $R_N^T\sub R_N^P$.  More specifically, one can easily verify that
\[
R_N^P\restr_T = R_N^T \qquad\text{for any $1\leq q\leq r$ and $N\normal\S_q$.}
\]
Since $T\cong\T_r$, Theorem \ref{thm:Malcev} gives
\[
\Cong(T) = \{\nab_T\} \cup \set{R_N^T}{1\leq q\leq r,\ N\normal\S_q}.
\]
We also abbreviate the notation for the Rees congruences on $P$ and $T$:
\[
R_q^S = R_{I_q^S} = \nab_{I_q^S}\cup\De_S \qquad\text{for each $0\leq q\leq r$, where again $S$ stands for either $P$ or $T$.}
\]
Note that $R_0^S=\De_S$ and $R_r^S=\nab_S$.  For $0\leq q<r$, we have $R_q^S=R_{\{\id_{q+1}\}}^S$.

\begin{lemma}\label{lem:RN}
If $r\geq2$, then $\nab_T^\sharp = \nab_P$, and $(R_N^T)^\sharp = R_N^P$ for all $1\leq q\leq r$ and $N\normal\S_q$.
\end{lemma}

\begin{proof}
We first prove that 
\begin{equation}\label{eq:RqT}
(R_q^T)^\sharp=R_q^P\qquad\text{for all } 0\leq q\leq r ,
\end{equation}
and we note that this includes $\nab_T^\sharp=\nab_P$.
Letting $\tau=(R_q^T)^\sharp$, it is clear that $\tau\subseteq R_q^P$, so it remains to show that $R_q^P\subseteq \tau$.
When $q=0$, both relations are $\De_P$, so we now assume that $q\geq1$.  
Note that $R_q^T(\sub\tau)$ contains a pair from $D_q^T\times D_1^T(\sub D_q^P\times D_1^P)$.
Since $P\star D_q^P\star P=I_q^P$ and $P\star D_1^P \star P=D_1^P$,
 it suffices to show that all elements of $D_1^P$ are $\tau$-related; this reasoning is used, and explained in more detail, in \cite[Lemma~2.4]{ER2023}.  Now,
\[
D_1^P = \set{c_x}{x\in X} \AND D_1^T = \set{c_x}{x\in \im(a)},
\]
where we write $c_x:X\to X$ for the constant map with image $\{x\}$.
Since $\nab_{I_q^T}\sub R_q^T$ and $q\geq1$, the elements of $D_1^T$ are all $\tau$-related.  Now let $x\in X$.  Recall that $a=\trans{A_1&A_2&\cdots&A_r\\a_1&a_2&\cdots&a_r}$, and without loss of generality assume that $x\in A_1$.  Keeping in mind $r\geq2$, we can complete the proof of \eqref{eq:RqT} by showing that $c_x\mr\tau c_{a_2}$.  To do so, let
$b = \trans{A_1&A_2&\cdots&A_r\\x&a_2&\cdots&a_r}$.
Since $a=aba$, it follows that $\rank(aba)=r=\rank(b)$, and so $b\in P$.  But from $(c_{a_1},c_{a_2})\in\tau$, it follows that $(c_x,c_{a_2}) = (c_{a_1}\star b,c_{a_2}\star b)\in\tau$, completing the proof of \eqref{eq:RqT}.

Now fix some $1\leq q\leq r$ and $N\unlhd \S_q$, and write $\si = (R_N^T)^\sharp$.  It is clear that
\begin{equation}\label{eq:tau1}
\si \sub R_N^P = \nab_{I_{q-1}^P} \cup \nu_N^P \cup \De_P,
\end{equation}
and we need to show that $R_N^P\subseteq \si$. From \eqref{eq:tau1} we have
\begin{equation}\label{eq:tau2}
\si = \si\restr_{I_{q-1}^P} \cup \si\restr_{D_q^P} \cup \De_P,
\end{equation}
We will complete the proof by showing that
\begin{equation*}
\si \supseteq \nab_{I_{q-1}^P}\AND
\si \supseteq \nu_N^P.
\end{equation*}
The first follows  from  \eqref{eq:RqT}, as $\si = (R_N^T)^\sharp \supseteq(R_{q-1}^T)^\sharp = R_{q-1}^P\supseteq\nab_{I_{q-1}^P}$.  For the second, we first observe from \eqref{eq:tau1} and \eqref{eq:tau2} that
$\si\restr_{D_q^P} \sub \nu_N^P \sub {\H^P}\restr_{D_q^P}$.
Since $D_q^P$ is stable and regular, it follows from Lemma~\ref{lem:28}, writing $e$ for any idempotent in $D_q^T(\sub D_q^P)$, that
\[
\si\restr_{D_q^P} = \si\cap{{\H^P}\restr_{D_q^P}} = \nu_{N'}^P, \WHERE N' = \set{g\in H_e^P}{(e,g)\in\si}.
\]
As explained just before Lemma \ref{lem:28}, we can also assume that $N\sub H_e^P(=H_e^T)$.
Now, for any $g\in N$ we have $(e,g)\in\nu_N^T\sub R_N^T\sub\si$, so that $g\in N'$. 
This shows that $N\sub N'$, and so $\nu_N^P\subseteq \nu_{N'}^P\subseteq\si$,  completing the proof.
\end{proof}

The equality $\nab_T^\sharp = \nab_P$ in Lemma \ref{lem:RN} does not hold for $r=1$.  Indeed, when $r=1$ we have $\nab_T=\De_T$ (cf.~Remark \ref{rem:1}), so that $\nab_T^\sharp=\De_P\not=\nab_P$ (unless $|X|=1$).  The $(R_N^T)^\sharp = R_N^P$ part of the lemma does hold for $r=1$, but simply says $\De_T^\sharp=\De_P$.

\subsection{Interactions between congruences and the $\widehat{\H}^P$-relation}\label{subsect:hHP}

\begin{lemma}\label{lem:fge}
Let $\si\in\Cong(P)$, and suppose $(f,g)\in\si\cap{\hH^P}$.  Then for any idempotent $e\in L_f^P$ we have $(f,g\star e) \in \si$ and $g\star e\in L_f^P \cap R_g^P$.  
\end{lemma}

\pf
From $e\in L_f^P$ and $f\mr\si g$, we have $f = f\star e \mr\si g\star e$.  
We also note that
\[
g \hH^P f \L^P e \implies g\hL^P e \implies \ol g\L^T\ol e \implies \ol g=\ol g\star\ol e=\ol{g\star e} \implies g \hH^P g\star e \implies g\D^P g\star e.
\]
It follows from stability that $g\R^P g\star e$.
Since $e\L^P f \hH^P g \D^P g\star e$, it follows that $g\star e \D^P e$, and this time stability gives $g\star e\L^Pe\L^Pf$.  The last two conclusions give $g\star e\in R_g^P\cap L_f^P$.\epf

\begin{lemma}
\label{lem:th0}
Let $\si\in\Cong(P)$, and suppose $\si\cap(H\times H')\not=\es$ for a pair of $\H^P$-classes $H$ and $H'$ contained in a common $\hH^P$-class.  Then for every $h\in H$, we have $(h,h')\in\si$, where $h'$ is the unique element of $H'$ with $\ol h'=\ol h$.
\end{lemma}

\pf
Fix some $(f,g)\in\si\cap(H\times H')$, and note that $H=H_f^P$ and $H'=H_g^P$.  Let $e$ be any idempotent in $L_f^P$.  Lemma \ref{lem:fge} tells us that $g\star e$ is in $L_f^P\cap R_g^P$ and is $\sigma$-related to both $f$ and $g$.  It therefore suffices to prove the current lemma in the case that $f$ and $g$ are $\R^P$- or $\L^P$-related.  We assume that $f\R^Pg$, with the other case being dual.

We distinguish two cases, depending on whether $\widehat{H}_f^P=\widehat{H}_g^P$ is a rectangular group or not.

\pfcase{1}
Suppose first that $\widehat{H}_f^P=\widehat{H}_g^P$ is a rectangular group.  Let $e$ and $e'$ denote the idempotents in $H_f^P$ and $H_g^P$ respectively.  Raising $(f,g)\in\sigma$ to an appropriately large power yields $(e,e')\in\sigma$.  Since these idempotents are $\hH^P$-related, we also have $\ol e=\ol e'$.  Since $h=h\star e$, it follows from Green's Lemma that the element $h' = h\star e'$ belongs to $H_{e'}^P=H_g^P$.  We also have $(h,h') = (h\star e,h\star e')\in\si$, and $\ol h = \ol h\star\ol e = \ol h\star\ol e' = \ol h'$.

\pfcase{2}
Now suppose $\widehat{H}_f^P=\widehat{H}_g^P$ is not a rectangular group.  Choose any $f'\in L_f^P$ whose $\hH^P$-class is a rectangular group, and pick any $b,c\in P$ such that $b\star f=f'$ and $c\star f'=f$.  By Green's Lemma the maps
\[
R_f^P\to R_{f'}^P : u\mt b\star u \AND R_{f'}^P\to R_f^P : v\mt c\star v
\]
are mutually inverse $\L^P$-preserving bijections.  In particular, the element $g'=b\star g$ is $\R^P$-related to~$f'$, and we have $(f',g')=(b\star f,b\star g)\in\sigma$ and $g'\L^Pg\hL^Pf\L^Pf'$, which together with $g'\R^Pf'$ implies $(f',g')\in{\hH^P}$.  Case 1 therefore applies, and tells us that the element $d = b\star h\in H_{f'}^P$ is $\si$-related to the unique element $d' \in H_{g'}^P$ with $\ol d = \ol d'$.  Now set $h' = c\star d'$, noting that this belongs to $H_g^P$.  We then have $(h,h') = (c\star d,c\star d')\in\si$, and $\ol h = \ol c \star\ol d = \ol c \star\ol d' = \ol h'$.
\epf

\section{\boldmath A subdirect decomposition of $\Cong(P)$}\label{sect:join}

In this section we prove part \ref{it:main1} of Theorem \ref{thm:main}, which is subsumed in the following.  Here we recall that $T\cong\T_r$, and analogously to the rank parameters in \eqref{eq:rank} we write 
\begin{equation}\label{eq:rkxi}
\rk(\xi) = \max\set{q}{R_q^T\sub\xi} \qquad\text{for $\xi\in\Cong(T)$.}
\end{equation}

\begin{thm}\label{thm:Psi}
For $r\geq2$, the mapping 
\[
\Cong(P)\rightarrow \Cong(T)\times [\Delta_P,\kappa]:\sigma\mapsto (\sigma\restr_T,\sigma\cap\kappa)
\]
is a subdirect embedding of lattices, with image
\[
\bigset{ (\xi,\theta)\in  \Cong(T)\times [\Delta_P,\kappa]}{ \rk(\xi)\leq \rk(\theta)}.
\]
\end{thm}

\pf
Since $T$ is a subsemigroup of $P$, and $\ka$ a congruence of~$P$, we have two well-defined mappings
\begin{equation}
\label{eq:XiTh}
\Xi:\Cong(P)\to\Cong(T):\si\mt\si\restr_T \AND \Th:\Cong(P)\to[\De_P,\ka]:\si\mt\si\cap\ka.
\end{equation}
By Proposition \ref{pr:sdp}, we can show that these induce the stated subdirect embedding by showing that:
\bit
\item $\Xi$ and $\Th$ are lattice surmorphisms (Lemma \ref{lem:sur}), and
\item $\ker(\Xi)\cap\ker(\Th) = \De_{\Cong(P)}$ (Corollary \ref{lem:inj}).
\eit
Lemmas \ref{lem:etaxizeth}, \ref{lem:xi} and \ref{lem:th} combine to show that the image of the embedding is as stated.
\epf

We now set off towards establishing these lemmas, beginning with the following key observation.  Throughout this section we assume that $r\geq2$, even though many of the results and proofs are valid (albeit trivial) for $r=1$.

\begin{prop}\label{prop:join}
For any $\si\in\Cong(P)$ we have $\si = \Xi(\si)^\sharp\vee\Th(\si)$.
\end{prop}

\pf
Throughout the proof we write $\xi=\Xi(\si) = \si\restr_T$ and $\th=\Th(\si)=\si\cap\ka$.
Since $\xi,\th\sub\si$, we of course have $\xi^\sharp\vee\th\sub\si$.  For the reverse inclusion, fix some $(f,g)\in\si$.  We must show that $(f,g)\in\xi^\sharp\vee\th$.  Since $\xi = \si\restr_T\in\Cong(T)$, we have
\begin{equation}\label{eq:siT}
\xi = \nab_T \OR \xi = R_N^T = \nab_{I_{q-1}^T} \cup \nu_N^T \cup \De_T \qquad\text{for some $1\leq q\leq r$, and some $N\normal\S_q$.}
\end{equation}
In the first case, we have $\xi^\sharp = \nab_T^\sharp = \nab_P$ by Lemma \ref{lem:RN}, so certainly $(f,g)\in\xi^\sharp\vee\th$.  For the rest of the proof, we assume that $\xi = R_N^T$, as in \eqref{eq:siT}, so that $\xi^\sharp = R_N^P$ by Lemma \ref{lem:RN}.  We now split the proof into cases, according to whether the pair $(\ol f,\ol g) \in \xi$ belongs to $\De_T$, $\nab_{I_{q-1}^T}$ or $\nu_N^T$; cf.~\eqref{eq:siT}.

\pfcase1
If $(\ol f,\ol g)\in\De_T$, then $\ol f=\ol g$, i.e.~$(f,g)\in\ka$.  But then $(f,g)\in\si\cap\ka = \th \sub\xi^\sharp\vee\th$.

\pfcase2
If $(\ol f,\ol g)\in\nab_{I_{q-1}^T}$, then $\rank(f)=\rank(\ol f)\leq q-1$, and similarly $\rank(g)\leq q-1$.  But then
\[
(f,g) \in \nab_{I_{q-1}^P} \sub R_N^P = \xi^\sharp \sub \xi^\sharp\vee\th.
\]

\pfcase3
Finally, suppose $(\ol f,\ol g)\in\nu_N^T$. Since $\nu_N^T\sub{\H^T}$, it follows that $\ol f \H^T \ol g$, i.e.~that $f\hH^P g$.
By Lemma \ref{lem:th0} (with $H=H_f^P$, $H'=H_g^P$ and $h=f$), we have $(f,f')\in\si$, where $f'$ is the unique element of $H_g^P$ with $\ol f=\ol f'$.  But this means that in fact $(f,f')\in\si\cap\ka = \th$.  We also have $(f',g)\in\si$ by transitivity, and we have $(f',g)\in{\H^P}$ (as $f'\in H_g^P$).  Since $(\ol f',\ol g)=(\ol f,\ol g)\in\nu_N^T$, it follows that $(f',g)\in\nu_N^P \sub R_N^P = \xi^\sharp$.  But then $f \mr\th f' \mr\xi^\sharp g$, so that $(f,g)\in\xi^\sharp\vee\th$.
\epf

\begin{rem}\label{rem:joincomp}
Examining the final line of the three cases above, we showed that in fact the pair $(f,g)\in\si$ belongs to $\th\circ\xi^\sharp$.  Thus, any congruence $\si\in\Cong(P)$ satisfies $\si=\Xi(\si)^\sharp\circ\Th(\si) = \Th(\si) \circ \Xi(\si)^\sharp$.  
\end{rem}

Proposition \ref{prop:join} has the following immediate consequence.

\begin{cor}\label{lem:inj}
$\ker(\Xi)\cap\ker(\Th)=\De_{\Cong(P)}$.  \epfres
\end{cor}

We now bring in the rank parameters associated to congruences from $\Cong(T)$ and $[\De_P,\ka]$, defined in \eqref{eq:rank} and \eqref{eq:rkxi}.

\begin{lemma}\label{lem:etaxizeth}
If $\si\in\Cong(P)$, then with $\xi=\Xi(\si)$ and $\th=\Th(\si)$ we have $\rk(\xi)\leq\rk(\th)$.
\end{lemma}

\pf
Write $q=\rk(\xi)$.  By Lemma \ref{lem:RN} we have $R_q^P = (R_q^T)^\sharp \sub \xi^\sharp \sub\si$.  It then follows that $\ka\cap R_q^P\sub\ka\cap\si=\th$, which gives $\rk(\th)\geq q=\rk(\xi)$.  
\epf

Our next goal is to establish a converse of Lemma \ref{lem:etaxizeth}.  Namely, we will show that if $\xi\in\Cong(T)$ and $\th\in[\De_P,\ka]$ satisfy $\rk(\xi)\leq\rk(\th)$, then with $\si=\xi^\sharp\vee\th\in\Cong(P)$ we have $\xi=\Xi(\si)$ and $\th=\Th(\si)$.  We prove the claims regarding $\xi$ and $\th$ in the following two lemmas; for the first, we do not in fact need the assumption $\rk(\xi)\leq\rk(\th)$:

\begin{lemma}\label{lem:xi}
If $\xi\in\Cong(T)$ and $\th\in[\De_P,\ka]$, then with $\si=\xi^\sharp\vee\th$ we have $\xi=\Xi(\si)$.
\end{lemma}

\pf
Recall that $\Xi(\si)=\si\restr_T$.
Certainly $\xi\sub\si\restr_T$.  For the reverse inclusion, suppose $(f,g)\in\si\restr_T$; we must show that $(f,g)\in\xi$.  By assumption we have $f,g\in T$ and $(f,g)\in\si=\xi^\sharp\vee\th$.  It follows that there is a sequence $f = f_0 \to f_1 \to\cdots\to f_k = g$, where each $(f_i,f_{i+1})\in\xi^\sharp\cup\th$.  Since $f,g\in T$, we have $f = \ol f = \ol f_0$ and $g = \ol g = \ol f_k$, so we can complete the proof that $(f,g)\in\xi$ by showing that $(\ol f_i,\ol f_{i+1})\in\xi$ for each $0\leq i<k$.  But for any such $i$, we have
\[
(\ol f_i,\ol f_{i+1}) = (a\star f_i\star a,a\star f_{i+1}\star a)\in\xi^\sharp\cup\th,
\]
as $\xi^\sharp$ and $\th$ are both compatible.  In fact, since $\ol f_i,\ol f_{i+1}\in T$, we have 
\[
(\ol f_i,\ol f_{i+1}) \in \xi^\sharp\restr_T \cup \th\restr_T  = \xi \cup \De_T = \xi,
\]
where we used Lemma \ref{la:reli}\ref{it:reli2} and $\th\restr_T\sub\ka\restr_T=\De_T$ in the second step.
\epf

\begin{lemma}\label{lem:th}
If $\xi\in\Cong(T)$ and $\th\in[\De_P,\ka]$ satisfy $\rk(\xi)\leq\rk(\th)$, then with $\si=\xi^\sharp\vee\th$ we have $\th=\Th(\si)$.
\end{lemma}

\pf
Recall that $\Th(\si)=\si\cap\ka$.
For the duration of the proof we write $q=\rk(\xi)\leq\rk(\th)$.  Since $\th\sub\si$ and $\th\sub\ka$, we certainly have $\th\sub\si\cap\ka$, so we are left to establish the reverse containment.  This is trivial in the case $q=r$, as then $\rk(\th)=r$, and so $\th= \ka \supseteq\si\cap\ka$.
Thus, for the rest of the proof we assume that $q<r$, and we fix some $(f,g)\in\si\cap\ka$; we must show that $(f,g)\in\th$.

Since $\rk(\xi)=q<r$, we have $\xi=R_N^T$ for some $N\normal\S_{q+1}$.  Since $(\ol f,\ol g)\in\si\restr_T=\Xi(\si)=\xi$ (by Lemma~\ref{lem:xi}), it follows from the form of $\xi=R_N^T$ that either $\ol f,\ol g\in I_q^T$ or else $\ol f,\ol g\not\in I_q^T$, and in the latter case we have $\rank(\ol f)=\rank(\ol g)$ and $\ol f\H^T\ol g$.  Since $\rank(f)=\rank(\ol f)$ and $\rank(g)=\rank(\ol g)$, it follows that either $f,g\in I_q^P$ or else $f,g\not\in I_q^P$, and in the latter case we have $\rank(f)=\rank(g)$ and~${f\hH^Pg}$.

\pfcase1  Suppose first that $f,g\in I_q^P$.  Together with the fact that $(f,g)\in\ka$, it follows that
\[
(f,g) \in \ka \cap R_q^P \sub \th,
\]
where in the last step we used the fact that $\rk(\th)\geq q$.

\pfcase2  Now suppose $f,g\not\in I_q^P$, and let $p=\rank(f)=\rank(g)$.  Since $(f,g)\in\si=\xi^\sharp\vee\th$, there is a sequence $f = f_0 \to f_1 \to\cdots\to f_k = g$, where each $(f_i,f_{i+1})\in\xi^\sharp \cup \th$.  Now consider the sequence of $\ol f_i$ maps: $\ol f = \ol f_0 \to \ol f_1 \to\cdots\to \ol f_k = \ol g$.  For each $i$ we have $(f,f_i)\in(\xi^\sharp\cup\th)^\sharp = \xi^\sharp\vee\th = \si$, so $(\ol f,\ol f_i)\in\si\restr_T = \xi = R_N^T$ by Lemma \ref{lem:xi}.  Since ${\rank(\ol f)=\rank(f)=p>q=\rk(\xi)}$, we have $p=\rank(\ol f_i) = \rank(f_i)$ for all $i$.  We also have $\ol f\H^T\ol f_i$, and hence $f\hH^Pf_i$ for all $i$.

For each $0\leq i\leq k$, let $h_i$ be the unique element of $H_{f_i}^P$ such that $\ol h_i=\ol f$.  Since $\ol f=\ol g$ (as $(f,g)\in\ka$), we have $h_0=f_0=f$ and $h_k=f_k=g$, so it suffices to show that $(h_i,h_{i+1})\in\th$ for all $0\leq i<k$.  This follows from Lemma \ref{lem:th0} in the case that $(f_i,f_{i+1})\in\th$.  Keeping $(f_i,f_{i+1})\in\xi^\sharp\cup\th$ in mind, it remains to consider the case in which $(f_i,f_{i+1})\in\xi^\sharp=R_N^P$.  Since $\rank(f_i)=\rank(f_{i+1})=p>q$, it follows from the form of $\xi^\sharp=R_N^P$ that $f_i \H^P f_{i+1}$, i.e.~that $H_{f_i}^P=H_{f_{i+1}}^P$.  It follows that $h_i=h_{i+1}$ in this case, so certainly $(h_i,h_{i+1})\in\th$.
\epf

Here is the final missing piece of the proof of Theorem \ref{thm:Psi}.

\begin{lemma}\label{lem:sur}
$\Xi$ and $\Th$ are lattice surmorphisms.
\end{lemma}

\pf
Surjectivity of $\Xi$ follows from Lemma \ref{la:reli}\ref{it:reli2}, which says that $\Xi(\xi^\sharp)=\xi$ for all $\xi\in\Cong(T)$.  Surjectivity of $\Th$ follows from the fact that $\Th(\th)=\th$ for all $\th\in[\De_P,\ka](\sub\Cong(P))$.

It remains to show that both maps are lattice morphisms.  To do so, let $\si_1,\si_2\in\Cong(P)$, and write $\xi_i=\Xi(\si_i)$ and $\th_i=\Th(\si_i)$ for $i=1,2$.  We need to show that
\ben\bmc2
\item \label{Zh1} $\Xi(\si_1\cap\si_2) = \xi_1\cap\xi_2$,
\item \label{Zh2} $\Xi(\si_1\vee\si_2) = \xi_1\vee\xi_2$,
\item \label{Zh3} $\Th(\si_1\cap\si_2) = \th_1\cap\th_2$,
\item \label{Zh4} $\Th(\si_1\vee\si_2) = \th_1\vee\th_2$.
\emc\een
Items \ref{Zh1} and \ref{Zh3} follow quickly from the fact that $\Xi$ and $\Th$ are defined in terms of intersections, namely $\Xi(\si)=\si\restr_T=\si\cap(T\times T)$ and $\Th(\si)=\si\cap\ka$.  

For \ref{Zh2} and \ref{Zh4}, we may assume without loss of generality that $\xi_1\sub\xi_2$, as $\Cong(T)$ is a chain.  Since then $\xi_1^\sharp\sub\xi_2^\sharp$, it follows that $\xi_1^\sharp\vee\xi_2^\sharp = \xi_2^\sharp =(\xi_1\vee\xi_2)^\sharp$.  Combining this with Proposition \ref{prop:join} gives
\begin{equation}\label{eq:si1si2}
\si_1\vee\si_2 = (\xi_1^\sharp\vee\th_1)\vee(\xi_2^\sharp\vee\th_2) 
= (\xi_1\vee\xi_2)^\sharp \vee (\th_1\vee\th_2),
\end{equation}
with $\xi_1\vee\xi_2\in\Cong(T)$ and $\th_1\vee\th_2\in[\De_P,\ka]$.  Item \ref{Zh2} now follows immediately from Lemma \ref{lem:xi}.  Next we note that
\begin{align*}
\rk(\xi_1\vee\xi_2) &= \rk(\xi_2) &&\text{as $\xi_1\sub\xi_2$}\\
&\leq \rk(\th_2) &&\text{by Lemma \ref{lem:etaxizeth}}\\
&\leq \rk(\th_1\vee\th_2) &&\text{by definition, as $\th_2\sub\th_1\vee\th_2$.}
\end{align*}
Item~\ref{Zh4} now follows from \eqref{eq:si1si2} and Lemma \ref{lem:th}
\epf

\begin{rem}\label{rem:Psi}
One can use Theorem \ref{thm:Psi} to give a schematic diagram of the lattice $\Cong(P)$.  First, we identify this lattice with its image in $\Cong(T) \times [\De_P,\ka]$, which we will denote by $\Lam$.  We then break up $\Lam$ into what we will call \emph{layers}.  Each such layer is a sublattice consisting of all pairs with a fixed first coordinate:
\[
\Lam_\xi = \set{(\xi,\th)}{\th\in[\De_P,\ka],\ \rk(\th) \geq q}\qquad\text{for $\xi\in\Cong(T)$ with $q=\rk(\xi)$}.
\]
Note that for such $\xi$ we have
\[
\Lambda_\xi\cong [\ka_q,\ka] ,\WHERE \kappa_q=\kappa\cap R_q^P.
\]
These layers are then stacked on top of each other in the order
\[
\Lam_{\De_T} < \Lam_{R_1} < \Lam_{R_{\S_2}} < \Lam_{R_2} < \cdots < \Lam_{\nab_T}.
\]
The stacking of two consecutive layers $\Lambda_{\xi_1}<\Lambda_{\xi_2}$ is such that every element $(\xi_2,\theta)$ of $\Lambda_{\xi_2}$ covers the corresponding element $(\xi_1,\theta)$ of $\Lambda_{\xi_1}$.
This is illustrated in Figure \ref{fig:lattice}, in the case $r=3$.  

Note that Figure \ref{fig:lattice0} shows a special case of this, when $X=\{1,2,3,4\}$, and $a=\trans{1&2&3&4\\1&2&3&3}$.  The red and blue vertices are included in both figures to show the matching of certain congruences of $\T_X^a$ (in Figure \ref{fig:lattice0}) with their corresponding pairs from $\Cong(T)\times[\De_P,\ka]$ (in Figure \ref{fig:lattice}).  Specifically, the blue and red vertices in Figure \ref{fig:lattice0} correspond to the bottom and top elements of the layers, i.e.~to the congruences $\xi^\sharp\vee\ka_q$ and $\xi^\sharp\vee\ka$, respectively, where $q=\rk(\xi)$.  See also Remark \ref{rem:lamrho} and Figure \ref{fig:lattice1}.
\end{rem}

\nc\ellipsebit[3]{\draw[fill=green!50] (#1,#2*\xx) ellipse (#3 and #3*\sca*\scaa);}

\nc\ellipsebottom[2]{\draw[fill=gray!50] (#1,#2*\xx) ellipse (4 and 4*\sca*\scaa);}
\nc\ellipsemiddle[2]{\draw[fill=gray!50] (#1,#2*\xx) ellipse (3 and 3*\sca*\scaa*\scaa);}
\nc\ellipsetop[2]{\draw[fill=gray!50] (#1,#2*\xx) ellipse (2 and 2*\sca*\scaa*\scaa);}
\nc\wallbottom[2]{\draw[fill=gray!30] (#1+1,#2*\xx) ellipse (3 and 3*\sca*\scaa*\scaa);
\fill[gray!30] (#1+1,#2*\xx) -- (#1+1+2,#2*\xx);
}

\begin{figure}[ht]
\begin{center}
\scalebox{0.86}{
\begin{tikzpicture} [scale=0.7]
\small 
\begin{scope}
\draw[fill=gray!30] (0,1) ellipse (2 and 4);
\draw (0,2) ellipse (1.5 and 3);
\draw (0,3) ellipse (1 and 2);
\foreach \x in {5,1,-1,-3} {\fill (0,\x)circle(.1); \draw[-{latex}] (-2.5,\x)--(-.5,\x);}
\foreach \x/\y in {
5/\ka=\ka_3,
1/\ka_2,
-1/\ka_1,
-3/\De_P=\ka_0
} 
{\node[left] () at (-2.5,\x) {$\y$};}
\end{scope}
\begin{scope}[shift={(14,-3)}]
\nc\xx{2}
\nc\sca{.4}
\nc\scaa{.7}
\draw[fill=gray!50] (0,0) ellipse (5 and 5*\sca);
\draw[fill=gray!30] (1,0) ellipse (4 and 4*\sca*\scaa);
\fill[gray!30] (-3,2)--(-3,0)--(5,0)--(5,2);
\draw (-3,2)--(-3,0) (5,0)--(5,2);
\draw[fill=gray!50] (1,0+\xx) ellipse (4 and 4*\sca*\scaa);
\draw[fill=gray!30] (1,+\xx) ellipse (4 and 4*\sca*\scaa);
\fill[gray!30] (-3,2+\xx)--(-3,0+\xx)--(5,0+\xx)--(5,2+\xx);
\draw (-3,2+\xx)--(-3,0+\xx) (5,0+\xx)--(5,2+\xx);
\draw[fill=gray!50] (1,0+\xx+\xx) ellipse (4 and 4*\sca*\scaa);
\draw[fill=gray!30] (2,0+\xx+\xx) ellipse (3 and 3*\sca*\scaa*\scaa);
\fill[gray!30] (-1,4+\xx)--(-1,2+\xx)--(5,2+\xx)--(5,4+\xx);
\draw (-1,4+\xx)--(-1,2+\xx) (5,2+\xx)--(5,4+\xx);
\draw[fill=gray!50] (2,0+\xx+\xx+\xx) ellipse (3 and 3*\sca*\scaa*\scaa);
\draw[fill=gray!30] (2,0+\xx+\xx+\xx) ellipse (3 and 3*\sca*\scaa*\scaa);
\fill[gray!30] (-1,4+\xx+\xx)--(-1,2+\xx+\xx)--(5,2+\xx+\xx)--(5,4+\xx+\xx);
\draw (-1,4+\xx+\xx)--(-1,2+\xx+\xx) (5,2+\xx+\xx)--(5,4+\xx+\xx);
\draw[fill=gray!50] (2,0+\xx+\xx+\xx+\xx) ellipse (3 and 3*\sca*\scaa*\scaa);
\draw[fill=gray!30] (2,0+\xx+\xx+\xx+\xx) ellipse (3 and 3*\sca*\scaa*\scaa);
\fill[gray!30] (-1,4+\xx+\xx+\xx)--(-1,2+\xx+\xx+\xx)--(5,2+\xx+\xx+\xx)--(5,4+\xx+\xx+\xx);
\draw (-1,4+\xx+\xx+\xx)--(-1,2+\xx+\xx+\xx) (5,2+\xx+\xx+\xx)--(5,4+\xx+\xx+\xx);
\draw[fill=gray!50] (2,0+\xx+\xx+\xx+\xx+\xx) ellipse (3 and 3*\sca*\scaa*\scaa);
\draw (5,10)--(5,12);
\foreach \x/\y in {
0/{(\De_T,\ka)},
2/{(R_1,\ka)},
4/{(R_{\S_2},\ka)},
6/{(R_2,\ka)},
8/{(R_{\A_3},\ka)},
10/{(R_{\S_3},\ka)},
12/{(\nab_T,\ka)}
} {\fill [red] (5,\x)circle(.15); \fill [white] (5,\x)circle(.1); \draw[-{latex}] (7,\x)--(5.25,\x); \node[right] () at (7,\x) {$\y$};}
\foreach \x/\y/\z in {
0/0/{(\De_P,\ka_0)},
2/2/{(R_1,\ka_1)},
2/4/{(R_{\S_2},\ka_1)},
4/6/{(R_2,\ka_2)},
4/8/{(R_{\A_3},\ka_2)},
4/10/{(R_{\S_3},\ka_2)}
} {\fill [blue] (-5+\x,\y)circle(.15); \fill [white] (-5+\x,\y)circle(.1); \draw[-{latex}] (-7,\y)--(-5.25+\x,\y); \node[left] () at (-7,\y) {$\z$};}
\foreach \x/\y/\z in {
2/0/{(\De_P,\ka_1)},
4/4/{(R_{\S_2},\ka_2)}
} {\fill (-5+\x,\y)circle(.1); \draw[-{latex}] (-3+\x,\y)--(-4.75+\x,\y); \node[right] () at (-3+\x,\y){$\z$};}
\end{scope}
\end{tikzpicture}
}
\caption{Structure of the lattice $\Cong(P)$ when $\rank(a)=3$, as discussed in Remark \ref{rem:Psi}. The left-hand side represents the interval $[\Delta_P,\kappa]$, and its distingushed congruences $\kappa_q=\ka\cap R_q^P$. The right-hand side indicates the stacking of layers.}
\label{fig:lattice}
\end{center}
\end{figure}

\section{\boldmath A direct decomposition of the interval $[\De_P,\ka]$}\label{subsect:De}

We have now decomposed $\Cong(P)$ as a subdirect product of $\Cong(\T_r)$ and the interval $[\De_P,\ka]$ in~$\Cong(P)$.  The lattice $\Cong(\T_r)$ is well understood, thanks to Theorem \ref{thm:Malcev}, so we now turn to the task of understanding the interval $\iK=[\De_P,\ka]$, thereby proving Theorem~\ref{thm:main}\ref{it:main2}:  

\begin{thm}
\label{th:dk}
The mapping
\[
[\De_P,\ka] \to [\De_P,\lam]\times[\De_P,\rho]: \th\mt(\th\cap\lam,\th\cap\rho)
\]
is a lattice isomorphism.  
\end{thm}

\begin{proof}
We apply Proposition \ref{pr:dp}, for which we need to verify the following:
\begin{itemize}
\item
$\lambda,\rho\in\Cong(P)$ (Lemma \ref{la:lamcon}), so that $[\Delta_P,\lambda]$ and $[\Delta_P,\rho]$ are well-defined intervals in $\Cong(P)$.
\item
The mappings $\Phi_\lambda: \iK\rightarrow [\Delta_P,\lambda]:\theta\mapsto \theta\cap\lambda$ and
$\Phi_\rho: \iK\rightarrow [\Delta_P,\rho]:\theta\mapsto \theta\cap\rho$ are lattice surmorphisms
(Lemma \ref{la:philsur}). 
\item
$\ker(\Phi_\lambda)\cap\ker(\Phi_\rho)=\Delta_\iK$ (Corollary \ref{la:capdel}).
\item
$\ker(\Phi_\lambda)\circ\ker(\Phi_\rho)=\nabla_\iK$ (Lemma \ref{la:kercirc}). \qedhere
\end{itemize}
\end{proof}

\begin{lemma}
\label{la:lamcon}
The relations $\lambda$ and $\rho$ are congruences on $P$.
\end{lemma}

\begin{proof}
We prove the statement for $\rho$, the one for $\lambda$ being dual.
Since $\kappa$ is a congruence and $\R^P$ a left congruence, it follows that
$\rho=\kappa\cap {\R}^P$ is a left congruence.
To prove that $\rho$ is right compatible, suppose $(f,g)\in\rho$ and $b\in P$.  Since $\rho\sub\ka$ it follows that $\ol f=\ol g$, i.e.~$afa=aga$.  Since $(f,g)\in{\R^P}$, we can fix an idempotent $e$ in the $\R^P$-class $R_f^P=R_g^P$, and we have $f=e\star f$ and $g=e\star g$.  But then
\[
f\star b = e\star f\star b = e(afa)b = e(aga)b = e\star g\star b = g\star b,
\]
and certainly $(f\star b,g\star b)\in\rho$, as required.
\end{proof}

\begin{lemma}
\label{la:philsur}
$\Phi_\lambda$ and $\Phi_\rho$ are lattice surmorphisms.
\end{lemma}

\begin{proof}
We prove the statement for $\Phi_\lambda$, and the one for $\Phi_\rho$ is dual.
That $\Phi_\lambda$ is well defined follows from Lemma \ref{la:lamcon}, and that it is surjective from the fact that it acts as the identity mapping on its image,~$[\Delta_P,\lambda]$. 
That $\Phi_\lambda$ respects $\cap$ is immediate from the definition:
\[
\Phi_\lambda(\theta_1\cap\theta_2)=\theta_1\cap\theta_2\cap\lambda=(\theta_1\cap\lambda)\cap(\theta_2\cap\lambda)=\Phi_\lambda(\theta_1)\cap \Phi_\lambda(\theta_2) \qquad\text{for $\th_1,\th_2\in[\De_P,\ka]$.}
\]
To prove that it also respects $\vee$, we need to verify that
\[
(\theta_1\vee\theta_2)\cap\lambda=(\theta_1\cap\lambda)\vee(\theta_2\cap\lambda) \qquad\text{for all $\th_1,\th_2\in[\De_P,\ka]$.}
\]
The reverse inclusion  is obvious. For the direct inclusion, let $(f,g)\in(\theta_1\vee\theta_2)\cap\lambda$.
This means that $(f,g)\in\lambda$ and there is a sequence $f=f_0\to f_1\to\dots\to f_k=g$ such that each
${(f_i,f_{i+1})\in\theta_1\cup\theta_2}$.
Let~$e$ be any idempotent in $L_f=L_g$.
Since $(f,f_i)\in\th_1\vee\th_2\sub\ka\sub{\hH^P}$, Lemma \ref{lem:fge} applies and tells us that all $f_i\star e$ are $\L^P$-related (to $f$), and we have ${(f_i\star e,f_{i+1}\star e)\in \theta_1\cup\theta_2}$ for $0\leq i<k$.
Hence $ (f_i\star e,f_{i+1}\star e)\in (\theta_1\cap\lambda)\cup(\theta_2\cap\lambda)$ for all such $i$, and therefore
\[
(f,g)=(f\star e,g\star e)=(f_0\star e,f_k\star e)\in (\theta_1\cap \lambda)\vee(\theta_2\cap\lambda).  \qedhere
\]
\end{proof}

\begin{lemma}
\label{la:capvee}
For every $\theta\in \iK$ we have $\theta=\Phi_\lam(\th)\vee\Phi_\rho(\th)$.
\end{lemma}

\begin{proof}
For the direct inclusion (the reverse is obvious), let $(f,g)\in\theta$, and fix some idempotent $e\in L_f^P$.
By Lemma \ref{lem:fge}, the element $h=g\star e$ belongs to $L_f^P\cap R_g^P$, and is $\theta$-related to both $f$ and $g$.
In particular, $(f,h)\in \theta\cap\lambda=\Phi_\lam(\th)$ and $(h,g)\in\theta\cap\rho=\Phi_\rho(\th)$, and so
$(f,g)\in \Phi_\lam(\th)\vee\Phi_\rho(\th)$.
\end{proof}

\begin{rem}\label{rem:joincomp2}
As with Remark \ref{rem:joincomp}, the above proof shows that $\theta=\Phi_\lam(\th)\circ\Phi_\rho(\th)=\Phi_\rho(\th)\circ\Phi_\lam(\th)$ for any $\th\in\iK$.  As a special case, $\ka=\lam\circ\rho=\rho\circ\lam$.
\end{rem}

\begin{cor}
\label{la:capdel}
$\ker(\Phi_\lambda)\cap \ker(\Phi_\rho)=\Delta_\iK$.  \epfres
\end{cor}

\begin{lemma}
\label{la:kercirc}
$\ker(\Phi_\lambda)\circ\ker(\Phi_\rho)=\nabla_\iK$.
\end{lemma}

\begin{proof}
Let $\theta_1,\theta_2\in \iK$ be arbitrary. Define $\theta=(\theta_1\cap\lambda)\vee(\theta_2\cap\rho)$.
Since $\Phi_\lambda$ is a lattice morphism (Lemma \ref{la:philsur}), we have
\[
\Phi_\lambda(\theta)=\Phi_\lambda(\theta_1\cap\lambda)\vee\Phi_\lambda(\theta_2\cap\rho)
=(\theta_1\cap\lambda\cap\lambda)\vee(\theta_2\cap\rho\cap\lambda)=(\theta_1\cap\lambda)\vee\Delta_P=
\theta_1\cap\lambda=\Phi_\lambda(\theta_1),
\]
and hence $(\theta_1,\theta)\in\ker(\Phi_\lambda)$.
Dually, $(\theta,\theta_2)\in\ker(\Phi_\rho)$, and so $(\theta_1,\theta_2)\in\ker(\Phi_\lambda)\circ\ker(\Phi_\rho)$.
\end{proof}

\section{\boldmath The intervals $[\De_P,\rho]$ and $[\De_P,\lambda]$ as subdirect products of full lattices of equivalence relations}\label{sec:RL}

We have just seen that the interval $\iK = [\De_P,\ka]$ in $\Cong(P)$ decomposes as a direct product of ${\iL = [\De_P,\lam]}$ and~$\iR = [\De_P,\rho]$, so we now turn our attention to these two intervals.  We treat them in essentially the same way, modulo some differing technicalities.  In the first subsection we give the full details for the interval $[\De_P,\rho]$, and in the second we indicate how to adapt this for $[\De_P,\lam]$.

\subsection[The interval ${[}\Delta_P,\rho{]}$]{\boldmath The interval $[\Delta_P,\rho]$}
\label{ss:R}

In what follows, we use the notation of Section \ref{sec:statement}, including the set $\cC_I$ of cross-sections of $\set{A_i}{i\in I}$, for $\es\not= I\sub[r]$.  
Every $\L^P$-class of $P$ takes the form $L_C=\set{ f\in P}{ \im(f)=C}$ where $C\in \cC_I$ for some $I$.  For a fixed $I$, the union $\bigcup_{C\in\cC_I} L_C$ is an $\hL^P$-class of $P$, and all $\hL^P$-classes have this form for some $I$.  For $\theta\in\iR = [\De_P,\rho]$ define
\begin{equation}
\label{eq:PsiI}
\Psi_I(\theta) = \bigset{ (C,C')\in\cC_I\times\cC_I}{ \theta\cap (L_{C}\times L_{C'})\neq\emptyset}.
\end{equation}

\begin{thm}
\label{thm:DeR}
The mapping
\[
 [\De_P,\rho]\rightarrow \prod_{\emptyset\neq I\subseteq [r]} \Eq(\cC_I):
\theta\mapsto (\Psi_I(\theta))
\]
is a subdirect embedding of lattices, with image
\[
\Bigset{ (\psi_I) \in \prod_{\emptyset\neq I\subseteq [r]} \Eq(\cC_I)}{ \psi_I\restr_J\subseteq \psi_J\text{ for all } \emptyset\neq J\subseteq I\subseteq [r]}.
\]
\end{thm}

\newpage

\begin{proof}
For the first assertion we apply Proposition \ref{pr:sdp}, for which  we need to establish that:
\bit
\item
each $\Psi_I$ is a well defined mapping $\iR\rightarrow \Eq(\cC_I)$  (Lemma \ref{la:Psiwell});
\item
each $\Psi_I$ is surjective (Lemma \ref{la:Psionto});
\item
each $\Psi_I$ is a lattice morphism (Lemma \ref{la:Psihom});
\item
$\bigcap_{I\subseteq [r]} \ker(\Psi_I)=\Delta_{\iR}$ (Corollary \ref{la:interkerPsi}).
\eit
We prove that the image of the  embedding is as stated in Lemmas \ref{la:Psidown} and \ref{la:reconstruct}.
\end{proof}

The first lemma establishes an equivalent, ostensibly stronger condition for membership of $\Psi_I(\th)$.

\begin{lemma}\label{lem:r}
Let $\emptyset\neq I\subseteq [r]$ and $C,C'\in\cC_I$.
For each $f\in L_{C}$, there is a unique element $f'\in L_{C'}$ such that $(f,f')\in\rho$.
Furthermore, for any $\theta\in\iR$, we have
\[
\th\cap(L_{C}\times L_{C'})\not=\es\IFF (f,f')\in\th \text{ for all } f\in L_{C}.
\]
\end{lemma}

\pf
For the first assertion, consider arbitrary elements $f\in L_{C}$ and $f'\in L_{C'}$. From the definition of $\rho=\kappa\cap{\R^P}$ we have
\[
(f,f')\in\rho \iff (f,f')\in\kappa \text{ and } f'\in R_f^P\cap L_{C'}.
\]
The set $R_f^P\cap L_{C'}$ is an $\H^P$-class contained in the $\hH^P$-class of $f$.
Both the existence and uniqueness of $f'$ now follow from $\kappa=\ker(\phi)$ and the fact that $\phi$ is bijective between any $\H^P$-class of $P$ and the corresponding $\H^T$-class of $T$.

For the second assertion, the reverse implication ($\Leftarrow$) is obvious.
For the direct implication ($\Rightarrow$) suppose $(g,h)\in\theta\cap (L_{C}\times L_{C'})$, and let $f\in L_{C}=L_g^P$ be arbitrary.
Let $b\in P$ be such that $b\star g=f$. By Green's Lemma we have $b\star h\in L_{C'}$, and $(f,b\star h)=(b\star g,b\star h)\in\theta$ since $\theta$ is a congruence.
From $\theta\subseteq \rho$ and $b\star h\in L_{C'}$,
it follows from the first assertion that $b\star h=f'$, and hence $(f,f')\in \theta$.
\epf

\begin{lemma}
\label{la:Psiwell}
Each $\Psi_I$ is a well-defined mapping $\iR\rightarrow \Eq(\cC_I)$, i.e.~for every $\theta\in\iR$ the relation $\Psi_I(\theta)$ is an equivalence on $\cC_I$.
\end{lemma}

\begin{proof}
Reflexivity and symmetry follow directly from the definition \eqref{eq:PsiI} of $\Psi_I$,
and the same properties for $\theta$. 
For transitivity, assume that $(C,C'), (C',C'')\in \Psi_I(\theta)$.
This means that $\theta\cap (L_{C}\times L_{C'})$ and $\theta\cap (L_{C'}\times L_{C''})$
are both non-empty. Let $f\in L_{C}$. By Lemma \ref{lem:r}, applied twice, there exists a unique $f'\in L_{C'}$ with $(f,f')\in \theta$, and then a unique $f''\in L_{C''}$ with $(f',f'')\in\theta$.
Transitivity of $\theta$ gives $(f,f'')\in \theta$, and hence $(C,C'')\in\Psi_I(\theta)$.
\end{proof}

\begin{lemma}
\label{la:Psihom}
Each $\Psi_I$ is a lattice morphism.
\end{lemma}

\begin{proof}
Let $\theta_1,\theta_2\in \iR$. We need to show that
\ben\bmc2
\item \label{PsiI1} $\Psi_I(\theta_1\cap\theta_2)=\Psi_I(\theta_1)\cap\Psi_I(\theta_2)$, and
\item \label{PsiI2} $\Psi_I(\theta_1\vee\theta_2)=\Psi_I(\theta_1)\vee\Psi_I(\theta_2)$.
\emc
\een
From \eqref{eq:PsiI} it is clear that $\Psi_I$ is monotone. Hence, $\Psi_I(\theta_1\cap\theta_2)\subseteq \Psi_I(\theta_i)\subseteq \Psi_I(\theta_1\vee\theta_2)$ for $i=1,2$.  This implies the direct inclusion in \ref{PsiI1} and the reverse inclusion in \ref{PsiI2}.

For the reverse inclusion in \ref{PsiI1}, suppose $(C,C')\in \Psi_I(\theta_1)\cap \Psi_I(\theta_2) $, so that ${\theta_i\cap (L_{C}\times L_{C'})\neq \emptyset}$ for $i=1,2$.
Let $f\in L_{C}$ be arbitrary, and let $f'\in L_{C'}$ be as in Lemma \ref{lem:r}. Then this lemma gives $(f,f')\in \theta_i\cap(L_{C}\times L_{C'})$ for $i=1,2$, and so $(C,C')\in \Psi_I(\theta_1\cap\theta_2)$

For the direct inclusion in \ref{PsiI2}, suppose $(C,C')\in \Psi_I(\theta_1\vee\theta_2)$, so that ${(\theta_1\vee\theta_2)\cap (L_{C}\times L_{C'})\neq\emptyset}$.  Fix some $(f,g)\in(\theta_1\vee\theta_2)\cap (L_{C}\times L_{C'})$.  So there is a sequence ${f=f_0\to f_1\to\dots\to  f_k=g}$, with each $(f_i,f_{i+1})\in\theta_1\cup\theta_2$.  Since $f_0,f_1,\dots,f_k$ are all $\hL^P$-related (as $\th_1,\th_2\sub\ka\sub{\hH^P}$), it follows that each $f_i\in L_{C_i}$ for some $C_i\in\cC_I$.  But then each $(\theta_1\cup\theta_2)\cap (L_{C_i}\times L_{C_{i+1}} )\neq \emptyset$, and so each ${(C_i,C_{i+1})\in \Psi_I(\theta_1)\cup \Psi_I(\theta_2)}$. It follows that $(C,C')=(C_0,C_k)\in \Psi_I(\theta_1)\vee \Psi_I(\theta_2)$.
\end{proof}

\begin{lemma}
\label{la:thetabigcup}
For every $\theta\in\iR$ we have
\[
\theta=\bigcup_{\emptyset\neq I\subseteq [r]} \th_I , \WHERE \th_I = \bigcup_{ (C,C')\in\Psi_I(\theta)}\rho \cap (L_{C}\times L_{C'}) .
\]
\end{lemma}

\begin{proof}
Throughout the proof we write $\th' = \bigcup_I \th_I$.

\pfitem{($\subseteq$)}
Suppose $(f,g)\in\theta$. Since $(f,g)\in{\hL^P}$, it follows that $f\in L_{C}$ and $g\in L_{C'}$ for some $\es\not= I\sub[r]$ and some $C,C'\in\cC_I$. 
So $(f,g)\in \theta\cap ( L_{C}\times L_{C'})$, meaning that $(C,C')\in \Psi_I(\theta)$, and $(f,g)\in\th_I\sub\th'$.

\pfitem{($\supseteq$)}
Suppose $(f,g)\in \theta'$, say with $(f,g)\in\th_I$.  So $(f,g)\in\rho\cap (L_{C}\times L_{C'})$ for some $(C,C')\in\Psi_I(\th)$, and it follows that $g=f'$ in the notation of Lemma \ref{lem:r}.  Since $(C,C')\in\Psi_I(\th)$ we have $\th\cap(L_C\times L_{C'})\not=\es$, and Lemma \ref{lem:r} gives $(f,g)=(f,f')\in\theta$.
\end{proof}

\begin{cor}
\label{la:interkerPsi}
$\displaystyle\bigcap_{\emptyset\neq I\subseteq [r]} \ker (\Psi_I)=\Delta_\iR$.  \epfres
\end{cor}

\begin{lemma}
\label{la:Psidown}
For every $\theta\in\iR$ and all $\emptyset\neq J\subseteq I\subseteq [r]$ we have $\Psi_I(\theta)\restr_J\subseteq \Psi_J(\theta)$.
\end{lemma}

\begin{proof}
Suppose $(B,B')\in\Psi_I(\theta)\restr_J$, so that $B=C\restr_J$ and $B'=C'\restr_J$, for some ${(C,C')\in \Psi_I(\theta)}$.  Write $C=\set{c_i}{ i\in I}$ and $C'=\set{c_i'}{ i\in I}$, where each $c_i,c_i'\in A_i$, and note that then ${B=\set{c_j}{ j\in J}}$ and $B'=\set{ c_j'}{ j\in J}$.
Since $(C,C')\in \Psi_I(\theta)$, there exists ${(f,g)\in\theta\cap (L_{C}\times L_{C'})}$, and we note that $\im(f)=C$ and $\im(g)=C'$.  Since $\ker(f)=\ker(g)$, as $(f,g)\in\th\sub\rho\sub{\R^P}$, we can therefore write $f=\binom{K_i}{c_i}$ and $g=\binom{K_i}{c_{i\pi}'}$ for some permutation $\pi\in\S_I$.  In fact, from $(f,g)\in\th\sub\ka=\ker(\phi)$ it follows that $\ol f=\ol g$, and from this that $\pi=\id_I$, so in fact $g=\binom{K_i}{c_i'}$.  
For each $j\in J$, let $j'\in I$ be such that $a_{j'}\in K_j$, and let $b$ be an arbitrary element of $P$ with $\im(b)=\set{a_{j'}}{j\in J}$.  
Then $(b\star f,b\star g)\in\th$, and we have $\im(b\star f)=B$ and $\im(b\star g)=B'$.  Thus, $(b\star f,b\star g)\in\th\cap(L_{B}\times L_{B'})$, and so $(B,B')\in\Psi_J(\th)$.  
\end{proof}

\begin{lemma}
\label{la:reconstruct}
Suppose for every $\emptyset\neq I\subseteq [r]$ we have an equivalence relation $\psi_I\in\Eq(\cC_I)$, and that these relations additionally satisfy $\psi_I\restr_J\subseteq \psi_J$ for all $J\subseteq I$.  Then the relation
\[
\theta=
\bigcup_I \theta_I, \qquad\text{where} \qquad \theta_I=\bigcup_{(C,C')\in\psi_I} \rho\cap (L_{C}\times L_{C'}),
\]
belongs to $\iR$, and we have $\Psi_I(\theta)=\psi_I$ for all $I$.
\end{lemma}

\begin{proof}
First we prove that $\theta$ is an equivalence relation on $P$, by showing that each~$\theta_I$ is an equivalence relation.
Reflexivity and symmetry follow directly from the same properties of the~$\psi_I$.
For transitivity, suppose $(f_1,f_2),(f_2,f_3)\in \theta_I$.
This certainly means that $(f_1,f_2),(f_2,f_3)\in\rho$, and hence $(f_1,f_3)\in\rho$.
Further, writing $C_i=\im(f_i)$ for $i=1,2,3$,   we have
$(C_1,C_2),(C_2,C_3)\in\psi_I$.
By transitivity of $\psi_I$ we have $(C_1,C_3)\in\psi_I$, from which we deduce $(f_1,f_3)\in\theta_I$.

Next, we check compatibility. Suppose $(f,g)\in\theta$ and $b\in P$.
Then, for some $I$ and $(C,C')\in\psi_I$, we have $(f,g)\in\rho\cap(L_{C}\times L_{C'})$, and we write $C=\set{ c_i}{ i\in I}$ and $C'=\set{ c_i'}{ i\in I}$.
As in the proof of Lemma \ref{la:lamcon} we have  $f\star b=g\star b$, and hence $(f\star b,g\star b)\in\theta$.
It remains to show that $(b\star f,b\star g)\in\theta$.
Certainly $(b\star f,b\star g)\in\rho$, since $\rho$ is a congruence.
As in the proof of Lemma \ref{la:Psidown}, we can write $f=\binom{K_i}{c_i}$ and $g=\binom{K_i}{c_i'}$.
Let $J=\set{ j\in I}{ \im(ba)\cap K_j\neq\emptyset}$.
Then $\im(b\star f)=\set{ c_j}{ j\in J}=C\restr_J$ and $\im(b\star g)=\set{ c_j'}{ j\in J}=C'\restr_J$.
By assumption, we have $(C\restr_J,C'\restr_J)\in \psi_I\restr_J\subseteq \psi_J$, and so $(b\star f,b\star g)\in \theta_J\subseteq\theta$,
as required. Thus, $\theta$ is a congruence, and it is clearly contained in $\rho$, so $\theta\in\iR$.

Finally, following the definitions of $\Psi_I$ and $\theta$, we have
\[
\Psi_I(\theta)=\bigset{ (C,C')\in \cC_I\times\cC_I}{ \theta\cap (L_{C}\times L_{C'})\neq\emptyset} =\psi_I.  \qedhere
\]
\end{proof}

\begin{lemma}
\label{la:Psionto}
Each $\Psi_I$ is surjective.
\end{lemma}

\begin{proof}
Suppose $\psi\in\Eq(\cC_I)$. For each $\emptyset\neq J\subseteq [r]$, define
\[
\psi_J=\begin{cases} \psi&\text{if } J=I\\ \nabla_{\cC_J} & \text{if } J\subset I\\ \Delta_{\cC_J}&\text{otherwise}.
\end{cases}
\]
This family satisfies the conditions of Lemma \ref{la:reconstruct}, and hence there exists $\theta\in\iR$ such that
$\Psi_J(\theta)=\psi_J$ for all $J$. In particular, $\Psi_I(\theta)=\psi$, completing the proof.
\end{proof}

\subsection[The interval ${[}\De_P,\lambda{]}$]{\boldmath The interval $[\De_P,\lambda]$}
\label{ss:L}

The interval $\iL=[\De_P,\lam]$ may be treated in an entirely analogous fashion
to $\iR=[\De_P,\rho]$, modulo some differing technical details regarding the $\restr$ operations.  
These arise because where for $\iR$ we needed to work with $\L^P$-classes contained in a common $\hL^P$-class, now we work with $\R^P$-classes contained in a common $\hR^P$-class.
In combinatorial terms, this translates to working with partitions of $[r]$, as opposed to subsets. 

We again use the notation of Section \ref{sec:statement}, including the sets $\cP_\bI$ (for $\bI\preceq\ldb r\rdb$) of all partitions of $[n]$ of the form $\bP=\set{ P_I}{ I\in \bI}$ such that $P_I\cap \im(a)=\set{ a_i}{ i\in I}$.
Every $\R^P$-class of $P$ is of the form $R_\bP=\set{f\in P}{ X/\ker(f)=\bP}$ where $\bP\in\cP_\bI$ for some $\bI\preceq \ldb r\rdb$.
For a fixed $\bI$, the union $\bigcup_{\bP\in\cP_\bI} R_\bP$ is a generic $\hR^P$-class of $P$.
For $\theta\in\iL$ define
\[
\Psi_\bI (\theta)=\bigset{ (\bP,\bP')\in \cP_{\bI}\times\cP_{\bI}}{ \theta\cap (R_{\bP}\times R_{\bP'})\neq\emptyset}.
\]

\begin{thm}
\label{thm:DeL}
The mapping
\[
[\De_P,\lambda]\rightarrow \prod_{\bI\preceq \ldb r\rdb} \Eq(\cP_\bI):
\theta\mapsto (\Psi_\bI(\theta))
\]
is a subdirect embedding of lattices, with image
\[
\Bigset{ (\psi_\bI) \in\prod_{\bI\preceq \ldb r\rdb} \Eq(\cP_\bI)}{ \psi_\bI\restr_\bJ\subseteq \psi_\bJ
\text{ for all } \bJ\preceq \bI\preceq \ldb r\rdb}.
\]
\end{thm}

\begin{proof}[\bf Sketch of proof.]
The proof follows exactly the same pattern as that of Theorem \ref{thm:DeR}.
Each of the results \ref{lem:r}--\ref{la:Psionto} has a straightforward left-right translation, and the proofs are also
relatively easy modifications; we omit the details. Put together, they prove the theorem, using Proposition \ref{pr:sdp}.
\end{proof}

\begin{rem}\label{rem:lamrho}
As in Remark \ref{rem:Psi}, one can use Theorems \ref{thm:DeR} and \ref{thm:DeL} to draw lattice diagrams for the intervals $[\De_P,\rho]$ and $[\De_P,\lam]$, by identifying these with their images under the embeddings from the relevant theorem.  We have done this in the special case that $X=\{1,2,3,4\}$ and ${a=\usebox{\transa}\in\T_X=\T_4}$, by calculating the $\cC_I$ and $\cP_\bI$ sets, and the appropriate systems of equivalences.  The (elementary) details of the calculation are omitted, but the resulting diagrams are shown in Figure \ref{fig:lattice1}.  Of course one could then construct a diagram for $[\De_P,\ka] \cong [\De_P,\lam]\times[\De_P,\rho]$.  We omit this step, as~$[\De_P,\ka]$ can be seen in Figure \ref{fig:lattice0} as the interval bounded by the solid red and blue vertices.  Examining the figures, one can check that this interval has size $90$, while $[\De_P,\rho]$ and $[\De_P,\lam]$ have sizes $6$ and $15$, respectively.  

Figure \ref{fig:lattice1} also shows the congruences $\lam_q=\lam\cap R_q^P$ and $\rho_q=\rho\cap R_q^P$, for $q=0,1,2,3$.  These can be used to construct the sub-intervals $[\ka_q,\ka] \cong [\lam_q,\lam]\times[\rho_q,\rho]$, and hence the `layers' $\Lam_\xi$ of $\Cong(P)$, discussed in Remark~\ref{rem:Psi}; cf.~Figures \ref{fig:lattice0} and \ref{fig:lattice}.
\end{rem}

\begin{figure}[t]
\begin{center}
\scalebox{0.7}{
\begin{tikzpicture}
\foreach \x in {0,1,3,4} {\draw[-{latex}] (-2.6,\x)--(-.25,\x);}
\foreach \x/\y in {
4/\rho=\rho_3,
3/\rho_2,
1/\rho_1,
0/\De_P=\rho_0
}
{\node[left] () at (-2.5,\x) {$\y$};}
\fill (0,0)circle(.1);
\fill (0,1)circle(.1);
\fill (-1,2)circle(.1);
\fill (1,2)circle(.1);
\fill (0,3)circle(.1);
\fill (0,4)circle(.1);
\draw
(0,0)--(0,1)--(-1,2)--(0,3)--(0,4) 
(0,1)--(1,2)--(0,3)
;
\begin{scope}[shift={(10,0)}]
\foreach \x in {0,5} {\draw[-{latex}] (-2.6,\x)--(-.25,\x);}
\draw[-{latex}] (-2.6,2.07)--(-.25,2.93);
\foreach \x/\y in {
5/\lam=\lam_3,
2/\lam_2,
0/\De_P=\lam_0=\lam_1
}
{\node[left] () at (-2.5,\x) {$\y$};}
\fill (0,0)circle(.1);
\fill (-1,1)circle(.1);
\fill (0,1)circle(.1);
\fill (1,1)circle(.1);
\fill (-1,2)circle(.1);
\fill (0,2)circle(.1);
\fill (1,2)circle(.1);
\fill (-2,3)circle(.1);
\fill (0,3)circle(.1);
\fill (1,3)circle(.1);
\fill (2,3)circle(.1);
\fill (-1,4)circle(.1);
\fill (0,4)circle(.1);
\fill (1,4)circle(.1);
\fill (0,5)circle(.1);
\draw
(0,0)--(-1,1)--(-1,2)--(0,1)--(0,0)--(1,1)--(1,2)--(0,1)
(-1,1)--(0,2)--(1,1)
(-1,2)--(-2,3)--(-1,4)--(0,3)--(-1,2)
(1,2)--(2,3)--(1,4)--(0,3)--(1,2)
(0,5)--(0,2)--(1,3)--(0,4)
(-1,4)--(0,5)--(1,4)
;
\end{scope}
\end{tikzpicture}
}
\caption{Left and right: the intervals $[\De_P,\rho]$ and $[\De_P,\lam]$ in the congruence lattice of $P=\Reg(\T_4^a)$, where $a=\usebox{\transa}$; cf.~Figure \ref{fig:lattice0}.}
\label{fig:lattice1}
\end{center}
\end{figure}

\section{Classification of congruences}\label{sect:class}

Our main result, Theorem \ref{thm:main}, describes the structure of the congruence lattice of $P=\Cong(\T_X^a)$, by successively decomposing the lattice via (sub)direct products.  The various technical results proved along the way allow us to give a transparent classification of the congruences themselves.  In order to make the classification succinct, we introduce some further notation.

A \emph{$\cC$-system} is a tuple $\Psi = (\psi_I) \in \prod_{\es\not= I\sub[r]}\Eq(\cC_I)$ satisfying $\psi_I\restr_J\sub\psi_J$ for all $\es\not=J\sub I\sub[r]$.  Given such a $\cC$-system $\Psi$, Lemma \ref{la:reconstruct} tells us that the relation
\[
\rho(\Psi) = \bigcup_{\es\not=I\sub[r]} \rho(\psi_I), \WHERE \rho(\psi_I) = \bigcup_{(C,C')\in\psi_I} \rho\cap (L_{C}\times L_{C'}),
\]
is a congruence of $P$.
We also define the parameter $\rank(\Psi) = \max\set{q}{\psi_I=\nab_{\cC_I} \text{ whenever $|I|\leq q$}}$.

A \emph{$\cP$-system} is a tuple $\Psi = (\psi_\bI) \in \prod_{\bI\preceq\ldb r\rdb}\Eq(\cP_\bI)$ satisfying $\psi_\bI\restr_\bJ\sub\psi_\bJ$ for all $\bJ\preceq\bI\preceq\ldb r\rdb$.  Dually, we have the associated congruence
\[
\lam(\Psi) = \bigcup_{\bI\preceq\ldb r\rdb} \lam(\psi_\bI), \WHERE \lam(\psi_\bI) = \bigcup_{(\bP,\bP')\in\psi_\bI} \lam\cap (R_{\bP}\times R_{\bP'}),
\]
and parameter $\rank(\Psi) = \max\set{q}{\psi_\bI=\nab_{\cP_\bI} \text{ whenever $|\bI|\leq q$}}$.

\begin{thm}\label{thm:class}
Let $X$ be a finite set, and let $a\in\T_X$ be an idempotent of rank $r$.  
Then every non-universal congruence $\si$ of $P=\Reg(\T_X^a)$ has a unique decomposition as
\[
\si = R_N^P \vee \lam(\Psi_1) \vee \rho(\Psi_2) = R_N^P \circ \lam(\Psi_1) \circ \rho(\Psi_2),
\]
where
\bit
\item $N$ is a normal subgroup of $\S_q$ for some $1\leq q\leq r$,
\item $\Psi_1$ is a $\cP$-system, and $\Psi_2$ a $\cC$-system, both of rank at least $q-1$.
\eit
\end{thm}

\pf
The $r=1$ case being trivial, we assume that $r\geq2$.
By Proposition \ref{prop:join} and Lemmas~\mbox{\ref{lem:etaxizeth}--\ref{lem:th}}, we have $\si=\xi^\sharp\vee\th$ for unique $\xi\in\Cong(T)$ and $\th\in[\De_P,\ka]$ with $\rank(\xi)\leq\rank(\th)$.  Since~$\si$ is non-universal, we have $\xi=R_N^T$ for some $1\leq q\leq r$ and $N\normal\S_q$, and then $\rank(\xi)=q-1$ and $\xi^\sharp=R_N^P$ by Lemma \ref{lem:RN}.  
By Lemma \ref{la:capvee} and Corollary~\ref{la:capdel} we have $\th=\th_1\vee\th_2$ for unique $\th_1\in[\De_P,\lam]$ and $\th_2\in[\De_P,\rho]$.  
By Lemmas \ref{la:thetabigcup} and \ref{la:reconstruct} we have $\th_2=\rho(\Psi_2)$ for a unique $\cC$-system $\Psi_2$, namely $\Psi_2=(\Psi_I(\th_2))_{\es\not= I\sub[r]}$.  Dually, we have $\th_1=\lam(\Psi_1)$ for a unique $\cP$-system $\Psi_1$.
We then have $q-1\leq\rank(\th)=\min(\rank(\Psi_1),\rank(\Psi_2))$.
Finally, we note that $\si = \xi^\sharp\vee\th_1\vee\th_2 = \xi^\sharp\circ\th_1\circ\th_2$ because of Remarks \ref{rem:joincomp} and \ref{rem:joincomp2}.
\epf

\section{\boldmath Application: the height of the lattice $\Cong(P)$}\label{sect:height}

The \emph{height} of a finite lattice $L$, denoted $\Ht(L)$, is the maximum size of a chain in $L$.  
Heights of  lattices of subgroups, subsemigroups and semigroup congruences have been treated in \cite{CST1989}, \cite{CGMP2017} and \cite{BEMMR2023}, respectively.
Results of \cite{BEMMR2023} include exact values for the heights of congruence lattices of semigroups satisfying the separation properties discussed in the introduction, and these do not hold for $P=\Reg(\T_X^a)$.
Nevertheless, we can compute the height of $\Cong(P)$ by using our (sub)direct decompositions from Theorem~\ref{thm:main}.

\begin{thm}\label{thm:height}
Let $X$ be a finite set of size $n$, and let $a=\trans{A_1&\cdots&A_r\\a_1&\cdots&a_r}\in\T_X$ be a mapping of rank~$r$.
Then the congruence lattice of $P=\Reg(\T_X^a)$ has height
\[
\Ht(\Cong(P)) = 3r + \prod_{q=1}^r (|A_q|+1) + \sum_{q=1}^r S(r,q)q^{n-r} - 2^r - B(r) - \begin{cases}
2 &\text{if $1\leq r\leq 3$,}\\
1 &\text{if $r\geq4$.}
\end{cases}
\]
\end{thm}

In this result, $B(r)$ stands for the Bell number and $S(r,q)$ the Stirling number of the second kind.
Looking at the formula,
we are unaware of a closed expression for the numbers $\sum_{q=1}^r S(r,q)q^{n-r}$, but we note that they appear as Sequence A108458 on the OEIS \cite{OEIS}.

We prove the theorem via a series of lemmas, in which we will repeatedly  make use of the well-known fact that
\begin{equation}\label{eq:htprod}
\Ht(L_1\times L_2) = \Ht(L_1) + \Ht(L_2) - 1 \qquad\text{for finite lattices $L_1$ and $L_2$.}
\end{equation}

\begin{lemma}\label{lem:Ht1}
We have $\Ht(\Cong(P)) = \Ht(\Cong(\T_r)) + \Ht[\De_P,\ka] - 1$.
\end{lemma}

\pf
This is trivial for $r=1$ (cf.~Remark \ref{rem:1}).  For $r\geq2$, we identify $\Cong(P)$ with its image~$\Lam$ under the embedding into  $\Cong(\T_r)\times[\De_P,\ka]$ from Theorem \ref{thm:main}\ref{it:main1}.  It then immediately follows that
\[
\Ht(\Cong(P)) \leq \Ht(\Cong(\T_r)\times[\De_P,\ka]) = \Ht(\Cong(\T_r)) + \Ht[\De_P,\ka] - 1.
\]
It remains to give a chain in $\Lam$ of the claimed size.  For this, we fix chains
\[
\De_{\T_r} = \xi_1 \subset\xi_2 \subset\cdots\subset \xi_k = \nab_{\T_r} \AND \De_P = \th_1\subset\th_2\subset\cdots\subset\th_l=\ka
\]
in $\Cong(\T_r)$ and $[\De_P,\ka]$, respectively, of length $k=\Ht(\Cong(\T_r))$ and $l=\Ht[\De_P,\ka]$.  It is then easy to verify that
\[
(\De_{\T_r},\De_P) = (\xi_1,\th_1) < (\xi_1,\th_2) < \cdots < (\xi_1,\th_l) = (\xi_1,\ka) < (\xi_2,\ka) <\cdots< (\xi_k,\ka) = (\nab_{\T_r},\ka)
\]
is a chain in $\Lam$ of the required length $k+l-1$.
\epf

It follows from  Theorem~\ref{thm:Malcev} (and the well-known classification of normal subgroups of symmetric groups) that
\begin{equation}\label{eq:HtTr}
\Ht(\Cong(\T_r)) = \begin{cases}
3r-2 &\text{for $1\leq r\leq 3$,}\\
3r-1 &\text{for $r\geq4$.}
\end{cases}
\end{equation}
We therefore turn to the task of finding an expression for $\Ht[\De_P,\ka]$.  
By Theorem \ref{thm:main}, we  have
\begin{equation}\label{eq:HtDeka}
\Ht[\De_P,\ka] = \Ht([\De_P,\lam]\times[\De_P,\rho]) = \Ht[\De_P,\lam] + \Ht[\De_P,\rho] - 1.
\end{equation}

\begin{lemma}
\label{la:HtDl}
$\Ht[\Delta_P,\lambda]=1-B(r)+\sum_{q=1}^rS(r,q)q^{n-r}$.
\end{lemma}

\begin{proof}
First we claim that
\begin{equation}
\label{eq:Lsum}
\Ht[\Delta_P,\lambda]=\Ht \Big(\prod_{\bI\preceq\ldb r\rdb} \Eq(\cP_\bI)\Big).
\end{equation}
The inequality $\leq$ follows from Theorem \ref{thm:main}\ref{it:main4}.
To establish the reverse inequality, we need to exhibit a chain in $[\Delta_P,\lambda]$ of size $\Ht(\prod_{\bI} \Eq(\cP_\bI))$.  We first observe that \eqref{eq:htprod} gives
\begin{equation}\label{eq:Lsum2}
\Ht\Big(\prod_{\bI}\Eq(\cP_\bI)\Big)=\sum_{\bI} \Ht(\Eq(\cP_\bI))-b+1=\sum_{\bI} |\cP_\bI|-b+1,
\end{equation}
where $b$ is the number of partitions $I\preceq\ldb r\rdb$, i.e.~$b=B(r)$, the Bell number.
We now list all the partitions of $[r]$ as $\bI_1,\dots,\bI_b$ extending the refinement partial order~$\preceq$ (i.e.~$\bI_i\preceq\bI_j \implies i\leq j$).
Then, for each $i$, pick a chain in $\Eq(\cP_{\bI_i})$ of length $l_i=\Ht(\Eq(\cP_{\bI_i})) = |\cP_{\bI_i}|$:
\[
\Delta= \psi_{i,1}<\psi_{i,2}<\dots<\psi_{i,l_i}=\nabla.
\]
Here $\Delta$ and $\nabla$ stand for $\Delta_{\cP_{\bI_i}}$ and $\nabla_{\cP_{\bI_i}}$, respectively.
We can find a copy of this chain in the image of $[\Delta_P,\lambda]$ in $\prod_{\bI} \Eq(\cP_\bI)$, as in Theorem \ref{thm:main}\ref{it:main4}, in the following way, where we continue to omit subscripts from various $\De$s and $\nab$s:
\begin{align*}
(\underbrace{\nabla,\dots,\nabla}_{i-1},\Delta,\Delta,\dots,\Delta) &=(\nabla,\dots,\nabla,\psi_{i,1},\Delta,\dots,\Delta) \\[-2mm]
&<(\nabla,\dots,\nabla,\psi_{i,2},\Delta,\dots,\Delta) \\
& \hspace{2mm} \vdots \\
&< (\nabla,\dots,\nabla,\psi_{i,l_i},\Delta,\dots,\Delta)=(\underbrace{\nabla,\dots,\nabla}_{i},\Delta,\dots,\Delta).
\end{align*}
Concatenating these chains for $i=1,\dots,b$ yields a chain of requisite length in $[\Delta_P,\lambda]$, and
establishes~\eqref{eq:Lsum}.

For a fixed partition $\bI\preceq\ldb r\rdb$, the set $\cP_\bI$ consists of all partitions $\bP=\set{P_I}{ I\in \bI}\preceq\ldb n\rdb$ with ${P_I\cap\im(a)=\set{ a_i}{ i\in I}}$.
If $|\bI |=q$ then $|\cP_\bI |=q^{n-r}$.
As there are $S(r,q)$ partitions of $[r]$ with $q$ blocks we conclude that
\begin{equation}\label{eq:Lsum3}
\sum_{\bI} |\cP_\bI|=\sum_{q=1}^r S(r,q)q^{n-r}.
\end{equation}
Putting together \eqref{eq:Lsum}, \eqref{eq:Lsum2} and \eqref{eq:Lsum3}, and remembering $b=B(r)$, completes the proof.
\end{proof}

\begin{lemma}
\label{la:HtDr}
$\Ht[\Delta_P,\rho]=1-2^r+\prod_{q=1}^r (|A_q|+1)$.
\end{lemma}

\begin{proof}
This is analogous to the previous lemma, and we just indicate the main points. To begin with:
\begin{align*}
\Ht[\Delta_P,\rho] &= \Ht\Big(\prod_{\emptyset\neq I\subseteq [r]} \Eq(\cC_I)\Big)&&\text{exactly as in Lemma \ref{la:HtDl}}
\\
&= \sum_I \Ht(\Eq(\cC_I)) -2^r+2 &&\text{as there are $2^r-1$ possible $I$}
\\
&=\sum_I |\cC_I |-2^r+2.&&
\end{align*}
The sum here is over all $\es\not=I\subseteq [r]$, and each $\cC_I$ consists of all cross-sections of $\set{ A_i}{ i\in I}$.
Thus $|\cC_I|=\prod_{i\in I} |A_i|$. The proof concludes with the observation that $\sum_I\prod_{i\in I}|A_i|=\prod_{q=1}^r(|A_i|+1)-1$.
\end{proof}

Theorem \ref{thm:height} now follows by combining Lemmas \ref{lem:Ht1}, \ref{la:HtDl} and \ref{la:HtDr} with equations \eqref{eq:HtTr} and \eqref{eq:HtDeka}.

\section{Concluding remarks}\label{sect:conc}

To conclude the paper, we discuss a number of natural directions for further study.  

First, one could try to classify the congruences of the variant $\T_X^a$ itself.  While this is certainly an appealing problem, it appears to be very challenging.  Indeed, while $\Reg(\T_4^a)$ has $271$ congruences for $a=\trans{1&2&3&4\\ 1&2&3&3}\in\T_4$, GAP calculations show that there are $21263$ congruences of $\T_3^b$ for $b=\trans{1&2&3\\1&2&2}\in\T_3$, and $3137$ \emph{principal} congruences of $\T_4^a$, i.e.~congruences of the form~$(f,g)^\sharp$, generated by the single pair $(f,g)$.  Moreover, there even exist such congruences $(f,g)^\sharp$ that do not relate any other non-trivial pairs; this happens when $af=ag$ and $fa=ga$.  In any case, understanding the entire lattice $\Cong(\T_X^a)$ does not seem feasible at present.  

Another enticing direction is to consider (regular) variants of other natural families of semigroups whose congruence lattices are already known, e.g.~linear monoids \cite{Malcev1953,DE2018}, infinite transformation monoids \cite{Malcev1952,Sandwich2} or diagram monoids~\cite{ER2022,EMRT2018,DDE2021}.  

It would also be interesting to examine the extent to which the methods of the current paper apply to more general semigroup variants, or even to sandwich semigroups in locally small categories \cite{Sandwich1,Sandwich2}.  The main challenge to overcome here is the fact that Mal'cev's classification of the congruences of the underlying semigroup $\T_X$ played a pivotal role in a number of our arguments above.

\footnotesize
\def\bibspacing{-1.1pt}
\bibliography{biblio}
\bibliographystyle{abbrv}

\end{document}